\documentclass[a4paper,11pt,reqno]{amsart}

\usepackage{ulem} 

\usepackage[margin=1in,includehead,includefoot,
            inner=20mm,
            outer=26mm,
            top=20mm,
            bottom=25mm,
            marginparsep=5mm,
            marginparwidth=20mm,
            ]{geometry}
\usepackage[utf8]{inputenc}
\usepackage[T1]{fontenc}
\usepackage{lmodern,microtype}
\usepackage{upref}
\usepackage{mathtools,amssymb,amsthm,amsmath}
\usepackage[foot]{amsaddr}
\usepackage{enumerate}
\usepackage{graphicx}
\usepackage{hyperref}
\usepackage{wrapfig}
\usepackage{array}
\usepackage[margin=1cm]{caption}
\usepackage{bookmark}
\usepackage{mycommands}
\usepackage{dsfont}
\usepackage{nicefrac}
\usepackage{comment}

\makeatletter
\let\old@setaddresses\@setaddresses
\def\@setaddresses{\bigskip{\parindent 0pt\let\scshape\relax\let\ttfamily\relax\old@setaddresses}}
\makeatother

\newtheorem{thm}{Theorem}
\newtheorem{corollary}[thm]{Corollary}
\newtheorem{prop}[thm]{Proposition}
\newtheorem{lem}[thm]{Lemma}
\theoremstyle{remark}

\newtheorem{rem}{Remark}
\theoremstyle{definition}
\newtheorem{example}{Example}
\newtheorem{dfn}{Definition}
\newtheorem{asu}{Assumption}


\hypersetup{
    pdftitle={}
    pdfauthor={Nicolas Zalduendo}
}

\newcommand{\onesp}{{\mathbf 1_p}}
\newcommand{\onesq}{{\mathbf 1_q}}

\begin{document}
    
    \title{The Multi-type Bisexual Galton-Watson Branching Process}
    \author{Coralie Fritsch$^1$, Denis Villemonais$^1$, Nicolás Zalduendo$^1$}
    \footnotetext[1]{Universit\'e de Lorraine, CNRS, Inria, IECL, UMR 7502, F-54000 Nancy, France}
    \email{coralie.fritsch@inria.fr; denis.villemonais@univ-lorraine.fr; zalduend1@univ-lorraine.fr}

    \begin{abstract}
    In this work we study the bisexual Galton-Watson process with a finite number of types, where females and males mate according to a ``mating function’’ and form couples of different types. We assume that this function is superadditive, which in simple words implies that two groups of females and males will form a larger number of couples together rather than separate. In the absence of a linear reproduction operator which is the key to understand the behaviour of the model in the asexual case, we build a concave reproduction operator and use a concave Perron-Frobenius theory to ensure the existence of eigenelements. Using this tool, we find a necessary and sufficient condition for almost sure extinction as well as a law of large numbers. Finally, we study the almost sure long-time convergence of the rescaled process  through the identification of a supermartingale, and we give sufficient conditions to ensure a convergence in $L^1$ to a non-degenerate limit.  
    \end{abstract}

    \maketitle

\section{Introduction}

    The single-type bisexual Galton-Watson branching process is a modification of the standard Galton-Watson process. It assumes that there exist two disjoint classes (sexes), males and females who together form mating units (or couples) which can accomplish reproduction. This process was first introduced by Daley in \cite{daley1968extinction}: it consists of a population where, in every generation $n=1,2,3,\dots$,  $F_n$ females and $M_n$ males form $Z_n=\xi(F_n,M_n)$ mating units, for $\xi$ a suitable and deterministic \textit{mating function}. Each mating unit reproduces independently of the others and with identical distribution giving birth to the new generation of males and females. 

    \begin{example}\label{example:mf} Some examples of mating functions are
        
        \begin{enumerate}
            
            \item $\xi(x,y)=x\min\{1,y\}$ called the \textit{promiscuous mating} model.
            \item $\xi(x,y)=\min\{x,y\}$ called the \textit{perfect fidelity mating} model.
            \item $\xi(x,y)=x$, which corresponds to the classical Galton-Watson model.
            
        \end{enumerate}
    \end{example}
    
    Daley studied in \cite{daley1968extinction} the properties for the first two mating functions of the previous example. Since Daley's work, extinction conditions have been studied for models with a more general family of superadditive mating functions (see for instance \cite{hull1982necessary}, \cite{bruss1984note}, \cite{daley1986bisexual}) and, in the last decades, results on the limit behaviour of this kind of processes were obtained (see for example \cite{alsmeyer1996bisexual}, \cite{alsmeyer2002asexual}, \cite{gonzalez1996limit}, \cite{gonzalez19972}). From these works, new models of two-sex populations have been developed, such as processes in random or varying environment (\cite{ma2006bisexual}, \cite{ma2009two}, \cite{molina2003bisexual}), processes with immigration (\cite{gonzalez2000limit}, \cite{gonzalez1996limit}, \cite{ma2006asymptotic}), processes with mating function depending on the number of couples (\cite{molina2002bisexual}, \cite{molina2006lalpha}, \cite{xing2005extinction}) and more recently, processes with random mating (\cite{jacob2017general}, \cite{molina2019population}). The interested reader can also consult the surveys of Hull \cite{hull2003survey}, Alsmeyer \cite{alsmeyer2002bisexual} and Molina \cite{molina2010two} for a wide description of the work accomplished on this family of processes.

\medskip

The asexual multi-type Galton-Watson branching process with a finite number of types (see \cite{harris1964theory}) is a discrete time Markov Chain on $\NN^p$, for $p$ a fixed integer, and can be thought as a system of particles where each particle is characterized by a type among $p$ options. Each particle reproduces independently and with an offspring distribution that depends only on its type, and if we define $X_{i,j}\in L^1$ the number of type $j$ particles produced by a type $i$ progenitor, then, under some assumptions on the process, the greatest eigenvalue of the matrix  $H_{i,j}=\EE(X_{i,j})$ determines if the process will be eventually extinct with probability $1$. 

\medskip

    Our focus of study is the multi-type bisexual process with a finite number of types. 
    Although specific models were studied in the two-sex population literature (see \cite{karlin1973criteria}, \cite{hull1998reconsideration}),  no general 
   mathematical description for a superadditive multi-type model has yet been established. The aim of this paper is to fill this gap.

We will consider a general model, which includes the natural extension of Daley's bisexual model to multi-type processes, where the vector of couples in the $n-$th generation is defined by 
         \begin{equation}
         \label{def:bGWbp}
         Z_n=(Z_{n,1},\dots,Z_{n,p})=\xi((F_{n,1},\dots,F_{n,n_f}),(M_{n,1},\dots,M_{n,n_m})),
         \end{equation}
where $\xi: \NN^{n_f}\times \NN^{n_m}\to \NN^p$ is a positive function such that $\xi(0,0)=0$ and $(F_{n,1},\dots,F_{n,n_f})$ and $(M_{n,1},\dots,M_{n,n_m})$ are the vectors of females and males respectively with $p,n_f,n_m\in \NN=\{0,1,2,\dots\}$ the number of types of couples, females and males respectively.  Each couple reproduces independently from the others and produces females and males according to a distribution that depends only on its type, such that
\[
        F_{n+1,j}=\sum_{i=1}^p\sum_{k=1}^{Z_{n,i}} X_{i,j}^{(k,n+1)},\text{ for }1\leq j\leq n_f,\ M_{n+1,j} =\sum_{i=1}^p\sum_{k=1}^{Z_{n,i}} Y_{i,j}^{(k,n+1)},\text{ for }1\leq j\leq n_m,
\]
where,  $(X^{(k,n)},Y^{(k,n)})_{k,n\in\NN}$ is a family of i.i.d. copies of $((X_{i,j})_{1\leq i \leq p, 1 \leq j \leq n_f},(Y_{i,j})_{1\leq i \leq p, 1\leq j \leq n_m})$
where  $X_{i,j}$ represents the number of female offspring of type $j$ produced by one couple of type $i$, and similarly for the males and $Y_{i,j}$. We recover Daley's process by setting $p=n_f=n_m=1$.

\medskip

    Since types may also encode gender, and to simplify our notation, we present in Section~\ref{sec:main.results} our definitions and main results in a more general setup, since they hold true for asexual and multi-sexual Galton-Watson processes (as they appear for instance for several plants species~\cite{BicknellKoltunow2004} and~\cite{BilliardLopez-VillavicencioEtAl2011,BilliardTran2012}). We present a law of large numbers (Theorem~\ref{thm:LGN1}), necessary and sufficient conditions for extinction (Theorem~\ref{thm:ext}) as well as asymptotic behaviour of the process (Theorems~\ref{thm:integrabilitycriterion} and \ref{thm:type_profiles}). We compare our results with existing works at the end of Section~\ref{sec:main.results}. Section~\ref{sec:LLN} is devoted to the proof of Theorem~\ref{thm:LGN1}. In Section~\ref{sec:eigenpbl}, we turn our attention to some extra properties, such as the eigenvalue problem  for concave functions, which play a fundamental role in the study of the extinction conditions, and prove Theorem~\ref{thm:ext}. Section~\ref{sec:proof.thm.integrabilitycriterion} is devoted to the proof of Theorem~\ref{thm:integrabilitycriterion}.
    Finally, we identify and establish properties on a supermartingale and prove Theorem~\ref{thm:type_profiles} in Section \ref{sec:supermartingale}.

\medskip

Notation: 
Unless otherwise stated, for $z\in (\RR_+)^p$, we denote by $z_i$ the i-th component of $z$.
We set $S:=\{z \in (\RR_+)^p, |z|=1\}$ to be the unitary ball for the $1-$norm on $(\RR_+)^p$, with $|z|=z_1+\dots+z_p$, and define $S^*$ as the elements on $S$ with strictly positive components. 
All random variables and random vectors are defined on the same probability space $(\Omega, \mathcal F, \PP)$, we denote $\EE$ for the expectation associated to $\PP$ and we let $L^1=L^1(\Omega,\mathcal F, \PP)$ be the set of integrable random variables defined on this space. 

\section{Model description and main results}
\label{sec:main.results}

\subsection{Model description}
For $p,q\in\NN^*=\{1,2,\dots \}$, we consider the process $(Z_n)_{n\in\NN}$ with values in $\NN^p$
 and the process $(W_n)_{n\in\NN}$ with values in $\NN^q$, where $Z_n$ and $W_n$ represent the mating units (of $p$ different types) and the individuals population (where individuals are of $q$ different types) respectively at the $n$-th generation.
We assume that at each generation $n\geq 1$, $Z_n$ is entirely determined by $W_n$, through a mating function $\xi: \NN^q\to \NN^p$ satisfying $\xi(0)=0$, that is
\[
    Z_n=(Z_{n,1},\dots,Z_{n,p})=\xi(W_{n,1},\dots,W_{n,q}),
\]
where the individuals population $W_n$ at the $n$-th generation is produced by the $Z_{n-1}$ mating units of the previous generation. Moreover, we assume that each mating unit reproduces independently from the others and is such that
\[
     W_{n,j}=\sum_{i=1}^p\sum_{k=1}^{Z_{n-1,i}} V_{i,j}^{(k,n)},\text{ for }1\leq j\leq q,
\]
with $(V^{(k,n)})_{k,n\in\NN}$ a family of i.i.d. copies of $V=(V_{i,j})_{1\leq i \leq p; 1\leq j \leq q}$, where $V_{1,\cdot},\dots, V_{p,\cdot}$ are $p$ mutually independent random vectors with values in $\NN^q$.  The random variable $V_{i,j}^{(k,n)}$ represents the number of offspring of type $j$ produced by the $k$-th mating unit of type $i$ of the $(n-1)$-th generation. Note that the offspring $V_{i,j_1}^{(k,n)}$ and $V_{i,j_2}^{(k,n)}$ of type $j_1$ and $j_2$ produced by the same mating unit are not necessary independent.

    Considering the empty sum as zero,  $(Z_n)_{n\in \NN}$ forms a discrete time Markov Chain on $\NN^p$ with absorbing state $0$. We define the probability of extinction with initial condition $z\in \NN^p$ as 
    \[
    q_z=\PP(\exists n\in \NN, Z_n=0|Z_0=z),
    \]
    and declare that the process will be almost surely extinct if $q_z=1$ for all $z\in \NN^p$.
    
    Although our model is sufficiently general to describe multi-type multi-sexual Galton-Watson branching process (see Example~ \ref{exa:bGWbp}),
    keeping in mind our motivation for the definition of this process, we call $(Z_n)_{n\in\NN}$ the multi-type bisexual Galton-Watson branching process (from now on, multi-type bGWbp).

    Throughout the paper, we make the assumption that the mating function $\xi$ of the process is superadditive, that is, 
        \begin{equation}
        \label{eq:superadditive}
        \xi(x_1+x_2)\geq \xi(x_1)+\xi(x_2),\ \forall x_1,x_2\in \NN^q.
        \end{equation}
    The intuition for this type of processes is that two populations form a bigger number of mating units together rather than separate. The idea of a superadditive bisexual Galton-Watson process for the single-type case was first introduced by Hull in \cite{hull1982necessary} and necessary and sufficient conditions for certain extinction were given by Daley et al. in \cite{daley1986bisexual}.

    \begin{rem}
        It will be useful in the proof to observe that, if we consider two superadditive functions $\xi_1,\xi_2$ such that $\xi_1(x)\leq \xi_2(x),\ \forall x\in \NN^q$, then a multi-type bGWbp with mating function $\xi_1$ is stochastically dominated from above by a process with the same offspring distribution but with mating function $\xi_2$. 
    \end{rem}

    \begin{example}[Multi-type bisexual Galton-Watson process] \label{exa:bGWbp}
    We recover the multi-type bGWbp setup presented in the introduction if we set $q=n_f+n_m$ and $(W_{n,1},\dots,W_{n,n_f})$ as the vector of females and $(W_{n,n_f+1},\dots,W_{n,n_f+n_m})$ as the vector of males of the $n-th$ generation. We can also extend this definition to a multi-sexual process by separating the set of possible types into a larger number of sexes.
    \end{example}

    We assume that all the random variables $(V_{i,j})_{1\leq i \leq p, 1\leq j \leq q}$ are integrable and define the matrix $\VV\in\RR^{p,q}$  by
            \begin{equation*}
            \VV_{i,j}=\EE(V_{i,j}),\quad\forall\, 1\leq i\leq p,\ 1\leq j\leq q.
            \end{equation*}
    We assume that, for all $j\in\{1,\ldots,q\}$, $\sum_{i=1}^p \VV_{i,j}>0$.

    We now define $\mathfrak M:\RR_+^p \longrightarrow (\RR_+\cup \{+\infty\})^p$ by
    \begin{equation}
    \label{eq:defM}
    \mathfrak M(z)=\lim\limits_{r\to+\infty} \dfrac{\xi(rz\VV)}{r}=\sup_{r\geq 1}\dfrac{\xi(rz\VV)}{r},
    \end{equation}
    where $\xi$ is any superadditive extension of $\xi$ to $\RR_+^q$ (see Remark~\ref{rem:extension.xi} below). The limit is well defined  and equal to the supremum according to Fekete's Lemma.

    As we will see, the convergence is in fact uniform on any compact subset of $S$ where $\mathfrak M$ is continuous (see Proposition~\ref{prop:M}). In addition $\mathfrak M$ is concave on $\RR_+^p$ (see Proposition~\ref{prop:post_hom_concave}). 

    \begin{rem}
    \label{rem:extension.xi}
        A superadditive function $\xi$ on  $\NN^q$ can always be extended to a superadditive function on  $\RR_+^q$ (for instance setting $x\in\RR_+^q \to \xi(\lfloor x\rfloor)$) and it will appear that $\mathfrak M$ does not depend on the choice of this extension (see Proposition~\ref{prop:M}).
    \end{rem}

    \begin{rem} 
    The functional $\mathfrak M$ plays a similar role as the reproduction matrix in the classical multi-type Galton-Watson case (see Example~\eqref{exa:GW} below).  However in our case, it is not necessarily linear, but only concave. In order to analyse the function $\mathfrak M$ and its iterates, we make use of concave Perron-Frobenius theory and, more precisely, of the results developed by Krause~\cite{krause1994relative} (see Section~\ref{sec:krause}).
    \end{rem}

\subsection{Main results}
    Our first main result  is a law of large numbers which relates $\mathfrak M$ with the behaviour of $Z$ in a large initial population setting and which is proved in Section~\ref{sec:proofLGN1}. We set $\mathfrak M^{n}=\mathfrak M\circ \cdots\circ \mathfrak M$ composed $n$ times with the convention that $\mathfrak M^{0}$ is the identity function, and, for all $i\in\{1,\ldots,p\}$,  $\mathfrak M_i$ is the $i^{th}$ component of $\mathfrak M$.

    \begin{thm}[Law of large numbers]
        \label{thm:LGN1} Let $(z_m)_{m\geq 1}$ be a random sequence in $\NN^p$and $z_\infty\in \RR_+^p\setminus\{0\}$ a deterministic value such that $z_m\sim_{m\to+\infty} mz_\infty$ almost surely. For all $m\geq 1$, denote by $(Z_n^m)_{n\geq 0}$ a multi-type bGWbp with common mating function and offspring distribution, but with initial configuration $Z_0^m=z_m$. Define $\mathfrak M$ as in \eqref{eq:defM} and assume it is finite over $S$. Then, for all $n\geq 0$,
        \begin{align*}
        Z_n^m\sim_{m\to+\infty} m\,\mathfrak M^n(z_\infty)\text{ almost surely}.
        \end{align*}
        If in addition  $(z_m/m)_{m\geq 1}$ is  independent of the random variables $V_{i,j}^{(k,n)}$ and uniformly integrable, then $Z_n^m/m$ converges to $\mathfrak M^{n}(z_\infty)$ in $L^1$. 
    \end{thm}

    As a  consequence, considering for $z\in \RR_+^p$ the sequence $z_m=\lfloor mz \rfloor $ for all $m\in \NN$, we have the following corollary (the second equality is a classical consequence for superadditive sequences).

    \begin{corollary}
    \label{cor:M_second_def} We have for all $z\in \RR_+^p$,
    \[\mathfrak      M(z) = \lim_{m\to +\infty}\dfrac{\mathbb E(Z_1\mid Z_0=\lfloor mz \rfloor )}{m}=\sup_{m\geq 1} \dfrac{\mathbb E(Z_1\mid Z_0=\lfloor mz \rfloor)}{m}.\]
    \end{corollary}

    The function $\mathfrak{M}$ extends to the multi-type case the \textit{mean growth rate} introduced in the single-type case by Bruss in \cite{bruss1984note} and used by Daley et al. in \cite{daley1986bisexual} to study the extinction conditions for the process.   Note also that, in the situation where $z_m=\lfloor mz\rfloor$, one can adapt the proof of Klebaner~\cite{Klebaner1993} to obtain convergence in law in Theorem~\ref{thm:LGN1}, as detailed in~\cite{Adam2016}.  However, almost sure and $L^1$ convergence obtained in Theorem~\ref{thm:LGN1} are needed in the proofs of the following results.

\medskip
    For the rest of the results of this section, we add the following transitivity and primitivity assumptions for the process.

    \begin{asu}
        \label{asu:bigasu}
        We assume that:
        \begin{enumerate}
            \item \label{asu:transience} The process is transitive, which means that
            \begin{equation*}
                \PP\left(\lim_{n\to\infty} |Z_n|\in \{0,+\infty\}\mid Z_0=z\right)=1,\quad\forall z\in \NN^p.
            \end{equation*}
            \item \label{asu:primitive} 
            The process is primitive, that is, for all $i\in \{ 1,\dots, p\}$, there exist $n_i,k_i\in \NN$ big enough such that for all $m\geq n_i$,
            \[
            \EE(Z_m|Z_0=k_ie_i)>0
            \]
            where $e_i$ is the $i-$th canonical vector in $\RR^p$.
        \end{enumerate}
    \end{asu}

    Assumption~\ref{asu:bigasu}.\ref{asu:transience} is a classical transitivity condition. The following cases are some examples where it holds.    
    \begin{enumerate}
        \item For all $i\in\{1,\dots,p\}$, $\PP(W_1=0|Z_0=e_i)>0$.
        \item In the context of the bisexual setting, $\forall i\in\{1,\dots,p\}$, $\PP(M_1=0|Z_0=e_i)>0$ and $\xi(x,0)=0,\ \forall x \in \RR^{n_f}$. 
        \item $\exists \ell \in \{1,\dots,p\}$, $\PP(|Z_1|=2|Z_0=e_\ell)>0$ and the process is strongly primitive, that is Assumption~\ref{asu:bigasu}.\ref{asu:primitive} is satisfied with $k_i=1$ for all $i\in \{1,\dots,p\}$ (see Appendix~\ref{sec:annex1} for details).
        Note that  the (not strong) primitivity of the process is not sufficient. In fact, choosing $q=p=2$ and 
        \[\xi(x,y)
             = \left(\left\lfloor\frac{y}{2}\right\rfloor,x\right),\quad
            V_1 \sim \frac{1}{2}\,\delta_{(2,0)}+\frac{1}{2}\,\delta_{(0,2)},\quad V_2\sim \delta_{(0,1)}.
        \]
        leads to a (not strongly) primitive branching process, satisfying $\PP(|Z_1|=2|Z_0=(1,0))=1/2>0$ and such that $\{(1,0),(0,2)\}$ is a recurrent class.
    \end{enumerate} 

    The following result provides a necessary and sufficient condition for almost sure extinction. This result is proved in Section~\ref{sec:proofMainTheorem}.

    \begin{thm}[Extinction criterion]\label{thm:ext} Assume that Assumption~\ref{asu:bigasu} holds and that $\mathfrak M$ is finite over $S$. Then there exist a unique $\lambda^* > 0$ and a unique $z^*\in S^*$ such that $\mathfrak M (z^*)=\lambda^* z^*$, and we have
        \[
        q_z =1, \forall z\in \NN^p \Longleftrightarrow \lambda^*\leq 1.
        \]
        If $\lambda^*>1$ or if there exists $z'\in \NN^p$ such that one of the components of $\mathfrak M(z')$ is not finite, then there exists $r>0$ such that, if $|z|>r$, then $q_z<1$.
    \end{thm}

Before turning to the next result, we point out that the last theorem encompasses the well known extinction criteria for the classical multi-type Galton-Watson process and single-type bisexual Galton-Watson process. Further examples are provided in Section~\ref{sec:examples}.

    \begin{example}[Multi-type Galton-Watson process]\label{exa:GW}
        The Galton-Watson case corresponds to the case where $p=q$ and the mating function is given by the identity function, $\xi(x)=x$, so that the process forms a classical asexual multi-type Galton-Watson process. In this case it is easy to see that $\mathfrak M(z)=z\VV$ and hence $\mathfrak M$ is a linear function and $\lambda^*$ is its greatest eigenvalue. We thus recover the well known fact that $\lambda^*\leq 1$ is a necessary and sufficient condition for certain extinction.
    \end{example}

    \begin{example}[Single-type bisexual Galton-Watson process]\label{exa:sbGW}
        The single-type bisexual Galton-Watson case corresponds to the case where $p=1$ and $q=2$. In this case it is easy to see that $\mathfrak{M}(z)=rz$ for some $r\geq 0$, that   $z^*=1$ and that $\lambda^*=r$. We  recover the fact that $\lambda^*\leq 1$ is a necessary and sufficient condition for certain extinction (see~\cite{daley1986bisexual}).
    \end{example}

    In the following theorem, we deal with the long-time behaviour of the process. In particular we prove that, on the non-extinction event $\{Z_n\neq 0,\ \forall n\geq 0\}$, the process will almost surely follow the direction of the eigenvector $z^*$. The proof of this theorem is in Section~\ref{sec:proof.thm.integrabilitycriterion}. We emphasize that, for this result, we do not assume any $L\log L$ type condition. As far as we know, this result is new even in the $p=1$ and $q=2$ (single type) case.

    In what follows, given $x,y\in \RR^p$, we denote
    \[|x,y|:=\{z\in \RR^p,\: x\leq z\leq y\}.\]
    \begin{thm}[Long time behaviour]
        \label{thm:integrabilitycriterion}
        Assume that Assumption~\ref{asu:bigasu} holds and that $\mathfrak M$ is finite over $S$. Let $\lambda^*$ and $z^*$ given by Theorem~\ref{thm:ext} and assume that $\lambda^*>1$. Then there exists $n_0\geq 1$ such that, for all $\varepsilon\in (0,1)$ and all $\eta\in (0,1)$, there exists $r>0$ such that
        \begin{align*}
        \mathbb P\left(Z_{n_0}\neq 0\text{ and }\forall n\geq n_0,\  Z_{n+1}\in|(1-\varepsilon)\mathfrak M(Z_n),(1+\varepsilon)\mathfrak M(Z_n)|\mid Z_0=z\right)\geq 1-\eta,\ \forall|z|\geq r.
        \end{align*}
        In addition,  on the non-extinction event $\{Z_n\neq 0,\ \forall n\geq 0\}$, and up to a $\mathbb P(\cdot\mid Z_0=z)$-negligible event, for all $k\geq 0$, 
        \begin{align*}
        \lim_{n\to+\infty} \frac{Z_{n+k}}{|Z_n|}=(\lambda^*)^kz^*.
        \end{align*}
    \end{thm}

    On the event of extinction, $\mathfrak M(Z_n)$ vanishes for $n$ large enough almost surely, which entails that $Z_{n+1}\in|(1-\varepsilon)\mathfrak M(Z_n),(1+\varepsilon)\mathfrak M(Z_n)|$ for all $n$ large enough. For $\lambda^*>1$ (so that extinction is not almost sure), this also holds true with probability one, as proved alongside Theorem~\ref{thm:integrabilitycriterion}. We thus obtain

    \begin{corollary}
    \label{cor:integrabilitycriterion}
    Assume that Assumption~\ref{asu:bigasu} holds and that $\mathfrak M$ is finite over $S$. Then, for all $\varepsilon \in (0,1)$ and all $z\in \NN^p$,
    \begin{align}
    \label{eq:cor}
    \mathbb P\left(\exists N\geq 0\text{ such that, } \forall n\geq N,\ Z_{n+1}\in|(1-\varepsilon)\mathfrak M(Z_n),(1+\varepsilon)\mathfrak M(Z_n)|\mid Z_0=z\right)=1.
    \end{align}
    \end{corollary}

    \begin{rem} As it will be clear from the proof, the value of $n_0$ in Theorem~\ref{thm:integrabilitycriterion} is in fact chosen deterministically as the minimal $n\in\NN$ such that $\mathfrak M^n(z)>0$ for all $z\in \NN^p\setminus \{0\}$, which exists thanks to Proposition~\ref{prop:primitivity} below.
    \end{rem}

    We now state a theorem related to the rescaled processes of mating units and children and we prove that they both have a non-negative limit with the same direction as the vector $z^*$ given by Theorem~\ref{thm:ext}. We also show that the event of extinction coincides (almost surely) with the event where this limit is equal to zero. This last part is well known in the classical branching case. As far as we know, it is new even in the $p=1$ and $q=2$ (single type) bisexual branching case. This theorem is proved in Section~\ref{sec:supermartingale}.

    \begin{thm}[Asymptotic profile]\label{thm:type_profiles}
        Assume that Assumption~\ref{asu:bigasu} holds, that $\mathfrak M$ is finite over $S$. Then, for all $z\in \NN^p$, there exists a real non-negative random variable $\mathcal C$ such that 
        \begin{equation}
        \label{eq:convergence_Zn}
        \dfrac{Z_n}{(\lambda^*)^n}\xrightarrow[n\to+\infty]{\PP(\cdot\mid Z_0=z)\ a.s.}\mathcal Cz^*\text{ and }\dfrac{W_n}{(\lambda^*)^n}\xrightarrow[n\to+\infty]{\PP(\cdot\mid Z_0=z)\ a.s.} \dfrac{1}{\lambda^*}\mathcal Cz^*\mathbb V,
        \end{equation}
       with $\lambda^*$ and $z^*$ given by Theorem~\ref{thm:ext}.
       
        Assume in addition that $\mathcal C$ is non-degenerate at $0$, which means that $\mathbb P(\mathcal C>0 \mid Z_0=z)>0$ for all $z\in \NN^p$ such that $q_z<1$. Then,  for all $z\in \NN^p$ and up to a $\mathbb P(\cdot \mid Z_0=z)$ negligible event,
    \[\{\mathcal C=0\}=\{\exists n\in\mathbb N,\:Z_n=0\}.\] 
    \end{thm}

    \begin{rem}
        We observe that the transitivity assumption implies that the condition $\mathbb P(\mathcal C>0 \mid Z_0=z)>0$ for all $z\in \NN^p$ such that $q_z<1$ is actually true if and only if $\mathbb P(\mathcal C>0 \mid Z_0=z)>0$ for some $z\in \NN^p$.
    \end{rem}

    A natural question that arises is to find conditions so that the previous convergence holds also in $L^1$ and the limit is non-degenerate at $0$. The following proposition, proved in Section~\ref{sec:proofpropL1-condition}, deals with this question, for which we consider the function $\mathcal P: \RR_+^p \longrightarrow \RR_+^p$ given by
    \begin{equation}
    \label{eq:function_P}
    \forall z\in \RR^p_+,\:\mathcal P(z)=\lim_{n\to+\infty} \dfrac{|\mathfrak M^n(z)|}{(\lambda^*)^n},
    \end{equation}
    which is well defined, according to Theorem~\ref{thm:eigenvalue} in Section~\ref{sec:krause} below. The condition we present is inspired by the work of González and Molina in the single-type case~\cite{gonzalez1996limit} and Klebaner's article~\cite{klebaner1985limit}.

    \begin{prop}\label{porp:L1-condition}
    Assume that Assumption~\ref{asu:bigasu} holds, that $\mathfrak M$ is finite over $S$ and that $\lambda^*>1$. If there exists a concave monotone increasing function $U:\mathbb R_+ \longrightarrow \mathbb R_+$, such that for all y $\in\RR_+$,
        \begin{equation}
        \label{eq:L1_condition}
        \sup_{z\in \RR_+^p: \mathcal P(z)=y} \mathbb E(|\mathcal P (Z_1) - \mathcal P(\mathfrak M(\lfloor z\rfloor ))|\mid Z_0=\lfloor z\rfloor )\leq U(y),
        \end{equation}
        with $y\to\nicefrac{U(y)}{y}$ non-increasing and
        \[\int\limits_1^{+\infty} \dfrac{U(y)}{y^2}\,\mathrm dy<+\infty,\]
        then the convergence in Theorem~\ref{thm:type_profiles} is in $L^1$ and the random variable $\mathcal C$ is non-degenerate at $0$.
    \end{prop}
    The existence of the function $U$ in the previous theorem may be difficult to check. In the following proposition we state sufficient conditions to ensure its existence, under a $V\log V$ condition and extra assumptions on the functions $\mathcal P, \xi$ and $\mathfrak M$. The proof of this proposition is in Section~\ref{sec:supermartingale_suff}.

    \begin{prop}\label{prop:suffconditions_convergence}
    Assume that $\mathfrak M$ is finite over $S$. In addition assume that both functions $\mathcal P$ and $\xi$ are Lipschitz, that $\mathbb E(V_{i,j}\log V_{i,j})<+\infty$ for all $i\in \{1,\dots ,p\}$, $j\in \{1,\dots,q\}$, and that there exists $\alpha>0$ such that
    \[\left|\dfrac{\xi(z\mathbb V)}{|z|} - \dfrac{\mathfrak M(z)}{|z|}\right|=O(|z|^{-\alpha}),\ \forall z\in\RR_+^p\setminus\{0\}.\]
    Then, the condition \eqref{eq:L1_condition} is satisfied.
    \end{prop}

    \begin{rem} Since originally $\xi$ is only defined over $\NN^q$, the statement ``$\xi$ is a Lipschitz function'' must be interpreted as ``there exists an extension of $\xi$ from $\NN^q$ to $\RR_+^q$, that is Lipschitz''.
    \end{rem}

    The previous conditions are not necessary conditions to ensure the existence of the function $U$ in \eqref{eq:L1_condition}. In fact, in Proposition~\ref{lem:cows_bulls}, we state that in the model of Example~\ref{exa:cows_bulls} below, for which there is no Lipschitz extension for $\xi$  over all $\RR_+^q$,  the $V\log V$ condition is sufficient to ensure the $L^1$ convergence to a non-degenerate random variable in \eqref{eq:convergence_Zn}.

\subsection{Examples in the context of bGWbp}
\label{sec:examples}
    The following examples are in the context of the multi-type bisexual Galton-Watson process. We recall that in this case $q=n_f+n_m$, where $n_f,n_m\in\NN^*$ are respectively the number of types for females and males. In order to be consistent with our notation, we write $\xi((x_1,\dots, x_{n_f}),(y_1,\dots,y_{n_m}))=\xi(x_1\dots,x_{n_f},y_1,\dots y_{n_m})$. In this context, for $i\in \{1,\dots p\}$ and $j\in \{1,\dots n_f\}$ we set $X_{i,j}=V_{i,j}$, and similarly, for $j\in \{1,\dots, n_m\}$, $Y_{i,j}=V_{i,n_f+j}$. 
    Finally, we define the matrices $\XX\in \RR^{p,n_f}$ and $\YY\in\RR^{p,n_m}$ given by
    \[\XX_{i,j} = \EE(X_{i,j}),\:\YY_{i,j}=\EE(Y_{i,j}).\]

    \begin{example}[Perfect fidelity mating function]
        Consider the case where $n_f=n_m=p$ and the mating function $\xi(x,y)=\min \{x,y\}:=\left(\min\{x_i,y_i\}\right)_{i\leq p}$, which is a natural extension of the perfect fidelity case presented by Daley (\cite{daley1968extinction}) to the multi-dimensional case.  In this case, we have
        \[
        \dfrac{\min\{kz\XX, kz\YY\}}{k}=\min\{z\XX, z\YY\},\ \forall k\geq 1.
        \]
         Hence the function $\mathfrak M$ takes the form $\mathfrak M(z)=\min\{z\XX, z\YY\}$. Let us discuss different particular instances of this model.

         \begin{itemize}
             \item  If $\XX\leq \YY$ are aperiodic irreducible non-negative matrices, then $\mathfrak M(z)=z\XX$ and so $\mathfrak M$ is a linear function and $\lambda^*$ is its greatest eigenvalue. In  the super-critical case (i.e. $\lambda^*>1$), the asymptotic profile of the types of the process in the non-extinction event is given by its positive left eigenvector. We thus observe that, despite the interaction between males and females, the extinction and growth characterization of the process is similar to the classical Galton-Watson case. 
            \item The case $\XX$ proportional with $\YY$ can be handled similarly: Let $(X_{i,j})_{1\leq i\leq p,1\leq j\leq p}$ and $(Y_{i,j})_{1\leq i\leq p,1\leq j\leq p}$ defined by
            \begin{align*}
            X_{i,j}=\sum_{\ell=1}^{U_{i,j}} \varepsilon_{i,j,\ell}\text{ and }Y_{i,j}=\sum_{i=\ell}^{U_{i,j}} (1-\varepsilon_{i,j,\ell}),
            \end{align*}
            where $(U_{i,j})_{1\leq i\leq p,1\leq j\leq p}$ is a random integrable array with mean $\mathbb U$ and $(\varepsilon_{i,j,\ell})_{1\leq i\leq p,1\leq j\leq p, \ell\in\NN}$ is an array of i.i.d. $\{0,1\}$ valued random variables independent from $U$.
            The variable $U_{i,j}$ describes the number of children of type $j$ from a mating unit of type $i$ and $\varepsilon_{i,j,\ell}$ determines if the $\ell$-th child is a female or a male. 
             Note that in this example $X_{i,j}$ and $Y_{i,j}$ are not independent. 
            Then, setting $\alpha=\mathbb P(\varepsilon_{i,j,\ell}=1)$, we have
            \begin{align*}
            \mathbb X=\alpha\mathbb U\text{ and }\mathbb Y =(1-\alpha)\mathbb U.
            \end{align*}
            As a consequence,
            \begin{align*}
            \mathfrak M(z)=\min\{\alpha,1-\alpha\}\,z\mathbb U
            \end{align*}
            and $\mathfrak M$ is a linear function.
            Assume now that $\mathbb U$ is an aperiodic irreducible non-negative matrix with greatest eigenvalue $\lambda_{\mathbb U}$ and positive left eigenvector $z_{\mathbb U}$. Then  $\lambda^*=\min\{\alpha,1-\alpha\}\lambda_{\mathbb U}$ and, in  the super-critical case, the asymptotic profile of the types of the process $(Z_n)_{n\in\NN}$ on the non-extinction event is given by $z^*=z_{\mathbb U}$.
            
            \item Let us now consider a non-linear case. Assume that $\mathbb X=\alpha \mathbf I_p+\beta \mathds{1}_p$ and $\mathbb Y=\alpha' \mathbf I_p+\beta' \mathds{1}_p$, where $\alpha,\alpha'\geq 0$ and $\beta,\beta'>0$ are  constants, $\mathbf I_p$ is the identity matrix of size $p\times p$ and  $\mathds{1}_p$ is the matrix of size $p\times p$ filled with ones. Then, for all $i\in\{1,\ldots,p\}$,
            \begin{align*}
            \mathfrak M_i(z)=\min\left\{\alpha z_i+\beta|z|,\alpha' z_i+\beta'|z|\right\}.
            \end{align*}
            Note that, for any permutation $\sigma$ of $\{1,\ldots,p\}$, we have
            \begin{align*}
            \mathfrak M(z_{\sigma(1)},\ldots,z_{\sigma(p)})=(\mathfrak M_{\sigma(1)}(z),\ldots,\mathfrak M_{\sigma(p)}(z)).
            \end{align*}
            Hence $z^*$ is stable by permutation of its components, so that $z^*=(\nicefrac1p,\ldots,\nicefrac1p)$. We deduce that
            \begin{align*}
            \lambda^*=|\mathfrak M(z^*)|=p\min\{\alpha/p+\beta,\alpha'/p+\beta'\}=\min\{\alpha+\beta p,\alpha'+\beta'p\}.
            \end{align*}
         \end{itemize}
    \end{example}

    \begin{rem} The previous example also covers a polygamous mating by one of the sexes, if we fix $d\in \NN$ and let $\xi(x,y) = \min \{x,dy\} = \left(\min \{x_i,dy_i\}\right)_{i\leq p}$, as Daley did \cite{daley1968extinction} in the single-type case. In this situation, we recover the same criterion.
    \end{rem}

    \begin{example}[Completely promiscuous mating function]
    \label{exa:cows_bulls}
        The case studied by Karlin and Kaplan~\cite{karlin1973criteria} corresponds to the case where the number of couples is equal to the number of females present in every generation (in particular this implies that $n_f=p$) given the condition that there is at least one male of each type present in every generation. In other words, they consider the mating function 
        \[
        \xi((x_1,\dots, x_p),(y_1,\dots , y_{n_m}))=(x_1,\dots ,x_p) \prod_{i=1}^{n_m} \mathds{1}_{\{y_i>0\}}.
        \]

        The function $\mathfrak M$ in this case corresponds to 
        \[
        \mathfrak M(z) =z\XX\mathds{1}_{\{z\YY>0\}}.
        \]
        We assume that $\XX$ is aperiodic irreducible and that $\forall j\leq n_m, \forall i\leq p\ : \YY_{i,j}>0$, (this last condition  ensures that Assumption~\ref{asu:bigasu} holds) then $z\YY>0$ for all $z\in \RR_+^p\setminus \{0\}$. In particular $\mathfrak M(z)=z\XX$ for all $z\in \RR_+^p$, which implies that the unique unitary  positive eigenvector of $\mathfrak M$ and its corresponding eigenvalue are the ones of $\XX$ given by the Perron-Frobenious Theorem. This result already appeared in~\cite{karlin1973criteria}.

    \end{example}
    
     In addition, Proposition~\ref{porp:L1-condition} entails the following original convergence property, proved in Section~\ref{sec:supermartingale_suff}. In particular, this model satisfies the conditions of Theorem~\ref{thm:type_profiles}.  As stated before, Example~\ref{exa:cows_bulls} does not satisfy the assumptions of Proposition~\ref{prop:suffconditions_convergence} since $\xi:\NN^q\to\NN^p$  is not Lipschitz.
    \begin{prop}\label{lem:cows_bulls} Consider the model in Example~\ref{exa:cows_bulls} above. Assume in addition that $\mathbb E(X_{i,j}\log X_{i,j})$ is finite for all $i,j \in \{1,\dots p\}$. Then, the rescaled process $\nicefrac{Z_n}{(\lambda^*)^n}$ converges almost surely and in $L^1$ to a non-degenerate random vector, with the same direction as $z^*$.
    \end{prop}

\section{Characterization of $\mathfrak M$ and proof of Theorem~\ref{thm:LGN1}}
\label{sec:LLN}    
    
    Let $(Z_n)_{n\in \NN}$ be a multi-type bGWbp with superadditive mating function $\xi$ and consider the function $\mathfrak M$ associated to this process given by \eqref{eq:defM}. We start by stating and proving some properties related to this function in Section~\ref{sec:charM} and prove Theorem~\ref{thm:LGN1} in Section~\ref{sec:proofLGN1}.
    
\subsection{Characterization of $\mathfrak M$}
\label{sec:charM}
In this section we give some fundamental properties of the operator $\mathfrak M$ defined in Section~\ref{sec:main.results}, and relate it to the behaviour of the number of mating units in the population. We start by proving that $\mathfrak M$ is concave and positively homogeneous, then we prove that it does not depend on the chosen extension for $\xi$ to $\RR^q$ and  state first properties of this function.

\begin{dfn} A function $F:\RR_+^p \longrightarrow \RR_+^p$ is said to be 
\begin{enumerate}
\item Concave if 
\[F(\alpha x + (1-\alpha) y) \geq \alpha F(x) + (1-\alpha) F(y),\]
for all $\alpha \in [0,1]$ and all $x,y\in \RR_+^p$.
\item Positively homogeneous if for all $\alpha>0$, $F(\alpha x)=\alpha F(x)$ for all $x\in \RR_+^p.$
\item Primitive if there exists $n_0\geq 1$ such that $F^m(x) >0$ for all $m\geq n_0$ and $x\in \RR_+^p \setminus \{0\}.$
\end{enumerate}
\end{dfn}

\begin{prop}\label{prop:post_hom_concave} The function $\mathfrak M$ is positively homogeneous and concave. 
\end{prop}

\begin{proof}
Let $\alpha>0$, then 
\[\mathfrak M(\alpha z)=\lim\limits_{k\to +\infty} \dfrac{\xi (\alpha k z\VV)}{k}=\alpha \lim\limits_{k\to +\infty}\dfrac{\xi (\alpha k z \VV)}{\alpha k}=\alpha \mathfrak M(z),\]
and so $\mathfrak M$ is positively homogeneous. Using this and the fact that $\xi$ (and hence $\mathfrak M$) is a superadditive function, we deduce that $\mathfrak M$ is a concave mapping.
\end{proof}

\begin{prop}
\label{prop:M}
For all $z\in \RR_+^p$, we have
        \begin{equation}
        \label{eq:proof1}
        \mathfrak M(z)=\lim\limits_{r\to+\infty} \dfrac{\xi(\lfloor rz \VV \rfloor)}{r}=\sup_{r\geq 1} \dfrac{\xi(\lfloor rz\VV\rfloor)}{r}.
        \end{equation}

In addition, for any compact set $K\subset S$ such that $\mathfrak M$ is continuous on $K$, $\mathfrak M$ is either bounded on $K$ or infinite on $K$. In the former case,
\begin{align}
\label{eq:propMunifconveq1}
\sup_{z\in K} \left|\mathfrak M(z) -\dfrac{\xi(rz\VV)}{r}\right|\xrightarrow[r\to+\infty]{}0,
\end{align}
and, in the latter case,
\begin{align}
\label{eq:propMunifconveq2}
\inf_{z\in K} \dfrac{\xi(rz\VV)}{r}\xrightarrow[r\to+\infty]{}+\infty.
\end{align}
\end{prop}

\begin{proof} For the first assertion consider $z\in \RR_+^p$, $u\in \{0,1\}^p$ given by $u_i=\mathds{1}_{\{(z\VV)_i\neq 0\}}$ and let $n\in \NN^*$ be such that $z\VV\geq \frac1n u$. Then, for all $r>0$,
        \[
        \dfrac{\xi(\lfloor (r+n)z\VV\rfloor)}{r+n}\geq \dfrac{\xi(\lfloor rz\VV +u \rfloor)}{r+n}
         \geq\dfrac{\xi(rz\VV)}{r}\dfrac{r}{r+n}.
        \]
        Taking the limit when $r\to +\infty$, we conclude that 
        \[
        \mathfrak M(z)\leq \lim\limits_{r\to+\infty} \dfrac{\xi(\lfloor rz \VV\rfloor)}{r}.
        \]
        The reverse inequality is direct using the fact that $\xi$ is non-decreasing in all its components, which concludes the proof of the first equality in~\eqref{eq:proof1}. The second equality is a consequence of Fekete's Lemma and the fact that $r\mapsto \xi(\lfloor rz \VV\rfloor)$ is superadditive.

        For the second part, let $K$ be a compact subset of $S$ such that $\mathfrak M$ is continuous on $K$. Note that, $\mathfrak M$ being continuous, it is either bounded or equal to $+\infty$ on $K$.
    
    We first consider the case where $\mathfrak M$ is bounded on $K$.   Since $\mathfrak M(z)\geq \frac{\xi(rz\VV)}{r}$ for all $z\in S$ and $r>0$, we only have to prove that
    \[
    \limsup_{r\to+\infty} \sup_{z\in K}\left(\mathfrak M(z)-\dfrac{\xi(rz\VV)}{r}\right)\leq 0.
    \]

    Assume the contrary. Then there exist $\varepsilon>0$ and two sequences $(z_n)_{n\in\mathbb N}\in K^{\mathbb N}$ and $(r_n)_{n\in\mathbb N}\in (0,+\infty)^{\mathbb N}$ such that $r_n\nearrow+\infty$ and
    \[
    \dfrac{\xi(r_nz_n\VV)}{r_n}\leq \mathfrak M(z_n)-\varepsilon.
    \]

    Since $K$ is compact, there exists, up to a subsequence, $z_\infty\in K$ such that $z_n\to z_\infty$. In particular, for all $\delta\in (0,1)$, there exists $n_{\delta,\varepsilon}$ such that, for all $n\geq n_{\delta,\varepsilon}$, $z_n\geq (1-\delta)z_\infty$ and $\mathfrak M(z_n)\leq \mathfrak M(z_\infty)+\varepsilon/2$ and hence
    \begin{align*}
    \dfrac{\xi\left((1-\delta)r_nz_\infty\VV\right)}{r_n}
    \leq \mathfrak M(z_\infty)-\varepsilon/2.
    \end{align*}

    By definition of $\mathfrak M$ and Proposition~\ref{prop:post_hom_concave}, the left hand side converges to $\mathfrak M((1-\delta)z_\infty)=(1-\delta)\mathfrak M(z_\infty)$ when $n\to+\infty$, and hence
    \begin{align*}
    (1-\delta)\mathfrak M(z_\infty)\leq \mathfrak M(z_\infty)-\varepsilon/2.
    \end{align*}

    Since this is true for all $\delta>0,\varepsilon>0$ and since $\mathfrak M(z_\infty)<+\infty$ by assumption, this is a contradiction. We thus proved~\eqref{eq:propMunifconveq1}.
    
    The proof of~\eqref{eq:propMunifconveq2} is similar. Assume that it does not hold true. Then there exist $A>0$ and two sequences $(z_n)_{n\in\mathbb N}\in K^{\mathbb N}$ and $(r_n)_{n\in\mathbb N}\in (0,+\infty)^{\mathbb N}$ such that $r_n\nearrow+\infty$ and
    \[
    \dfrac{\xi(r_nz_n\VV)}{r_n}\leq A.
    \] 

    This implies that, up to a subsequence, $z_n$ converges to $z_\infty\in K$ and that, for any $\delta\in (0,1)$,
    \begin{align*}
    (1-\delta)\mathfrak M(z_\infty)\leq A.
    \end{align*}
    But $\mathfrak M$ is equal to $+\infty$ on $K$ and hence we obtained $+\infty\leq A$, which is a contradiction. This concludes the proof of Proposition~\ref{prop:M}.
\end{proof}

As a consequence of Proposition~\ref{prop:M}, we have the following result in the case where $\mathfrak M$ is continuous.

\begin{corollary}
\label{cor:unifconv_M}
Assume $\mathfrak M$ is finite and continuous over $S$. Then, we have
\[\left|\dfrac{\mathfrak M(z)}{|z|}- \dfrac{\xi(z\mathbb V)}{|z|}\right|\xrightarrow[]{|z|\to +\infty}0.\]
\end{corollary}

The following lemma, used in the proof of the forthcoming Proposition~\ref{prop:primitivity} and Proposition~\ref{prop:anotherproponM}, relates the iterations of $\mathfrak M$ with the expectation of the process. It is proved in the Section~\ref{sec:proofLGN1}, after the proof of Theorem~\ref{thm:LGN1} and Corollary~\ref{cor:M_second_def}.  Note that the proof of this lemma is based on Corollary~\ref{cor:M_second_def}, however neither Theorem~\ref{thm:LGN1} nor Corollary~\ref{cor:M_second_def} makes use of Lemma~\ref{lem:domM},  Proposition~\ref{prop:primitivity} or Proposition~\ref{prop:anotherproponM}.

\begin{lem}
    \label{lem:domM}  For all $n\geq 0$ and all $z\in\mathbb N^p$, we have
    \begin{align*}
    \mathfrak M^n(z)\geq \EE(Z_n|Z_0=z),\ \forall n\in \NN.
    \end{align*}
\end{lem}

The following result states that the primitivity of the bGWbp entails the primitivity of $\mathfrak{M}$.

\begin{prop}\label{prop:primitivity}  Assume that Assumption~\ref{asu:bigasu}.\ref{asu:primitive} holds, that is, $(Z_n)_{n\in \NN}$ is primitive, then $\mathfrak M$ is a primitive function. 
\end{prop}

\begin{proof}
Since $(Z_n)_{n\in \NN}$ is primitive, we can find $N,k\in \NN$ big enough so that for all $i\in \{1,\dots,p\}$ and $m\geq N$ we have that $\EE(Z_m|Z_0=ke_i)>0$. Hence, for $m\geq N$ and $z\in \NN^p\setminus\{0\}$, using Lemma~\ref{lem:domM} and the superadditivity of $\xi$ and then of $z\mapsto \EE(Z_m|Z_0=z)$,
       \[
        k\mathfrak M^m(z)=\mathfrak M^m(kz)\geq \EE(Z_m|Z_0=kz)\geq \sum_{i=1}^p z_i \EE(Z_m|Z_0=ke_i)>0,
        \]
       and so $\mathfrak M^m(z)>0$, which concludes the proof.
\end{proof}

We finish this subsection by stating one last property on $\mathfrak M$.

\begin{prop}
     \label{prop:anotherproponM}
     Assume that Assumption~\ref{asu:bigasu}.\ref{asu:primitive} holds. We have $\inf_{z\in S} |\mathfrak M(z)|>0$ and, for all compact subset $K\subset S^*$ and for all $i\in\{1,\ldots,p\}$,  $\inf_{z\in K} \mathfrak M_i(z)>0$. 
\end{prop}
\begin{proof}
We start by proving the first assertion. Since $\mathfrak M$ is primitive by Proposition~\ref{prop:primitivity}, there exists $n_0\geq 1$ such that $
\mathfrak M^{n_0}(e_i)>0$ for all $i\in\{1,\ldots, p\}$. In particular, $\mathfrak M(e_i)\neq 0$ for all $i\in\{1,\ldots,p\}$. Using the concavity of $\mathfrak M$, we deduce that
\begin{align*}
\inf_{z\in S} |\mathfrak M(z)|\geq \inf_{z\in S}\left|\sum_{i=1}^p z_i \mathfrak M(e_i)\right|\geq \min_{i\in \{1,\ldots,p\}} |\mathfrak M(e_i)|>0.
\end{align*}
Let us now prove the second assertion. For $z\in K$, if $\mathfrak M_i(z)=+\infty$, by Proposition~\ref{prop:post_hom_concave}, $\mathfrak M_i(z')=+\infty$ for all $z' \in K$ and the result follows directly. If $\mathfrak M_i(z) <+\infty$, since $\mathfrak M_i$ is concave, it is locally Lipschitz on $S^*$ and hence $z\mapsto  \mathfrak M_i(z)$ is continuous on the compact set $K$. 
 It is thus sufficient to prove the result for any fixed $z\in S^*$. For this, we simply observe that, for any two $z,z'\in S^*$, we have, using the fact that $z\mapsto \mathfrak M_i(z)$ is positively homogeneous and increasing,
\begin{align*}
\mathfrak M_i(z) \geq \frac{\min_{j\in\{1,\ldots,p\}} z_j}{\max_{j\in\{1,\ldots,p\}} z'_j}  \mathfrak M_i(z').
\end{align*}
Hence $z\mapsto  \mathfrak M_i(z)$ is either null or positive on $S^*$. Since $\mathfrak M$ is primitive by Proposition~\ref{prop:primitivity}, $z\mapsto  \mathfrak M_i(z)$ is not null (take for instance $z=\mathfrak M^{n_0}(e_i)$), which concludes the proof.
\end{proof}

\subsection{Proof of Theorem~\ref{thm:LGN1}}
\label{sec:proofLGN1}
    We start with the following lemma, where we do not assume that $\mathfrak M$ is finite over $S$.
    \begin{lem}\label{lem: M(z)} Let $(z_k)_{k\geq 0}$ be a random sequence in $\NN^p$ such that $z_k\sim_{k\to+\infty} kz_\infty\in\RR_+^p$ almost surely. 
    We have 
    \[
    \dfrac{1}{k}\xi\left(\sum_{i=1}^p\sum_{m=1}^{z_{k,i}}V_{i,\cdot}^{(m)}\right)\longrightarrow \mathfrak M(z_\infty),
    \]
    almost surely when $k\to +\infty.$
    \end{lem}

    The proof of this lemma is inspired by~\cite{daley1986bisexual}.

    \begin{proof}         
    Since all the variables $V_{i,j}$ are integrable, then thanks to the strong law of large numbers, we have that for all $i\leq p, \ j\leq q$, 
    \[\dfrac{1}{n}\sum_{m=1}^{n} V_{i,j}^{(m)} \xrightarrow[n\to+\infty]{a.s.} \mathbb V_{i,j}.\]
    
    Assume first that $z_{\infty,i}>0$. In this case, since $z_{k,i}\to+\infty$ almost surely we deduce that 
    \[
    \dfrac{1}{z_{k,i}}\sum_{m=1}^{z_{k,i}} V_{i,j}^{(m)} \xrightarrow[k\to+\infty]{a.s.} \VV_{i,j}
    \]
    and
    hence
    \[
    \dfrac{1}{kz_{\infty,i}}\sum_{m=1}^{z_{k,i}} V_{i,j}^{(m)} \xrightarrow[k\to+\infty]{a.s.} \VV_{i,j}.
    \]
    Fix $0<\varepsilon<\min\limits_{i,j/ \VV_{i,j}\neq 0} \VV_{i,j}$. Hence, with probability one there exists $k_0$ (random) such that if $k\geq k_0$, then 
    \[
    kz_{\infty,i}(\VV_{i,j}-\varepsilon)\leq \sum_{m=1}^{z_{k,i}} V_{i,j}^{(m)}\leq kz_{\infty,i}(\VV_{i,j}+\varepsilon).
    \]
    
    Assume now that $z_{\infty,i}=0$. Then, almost surely, there exists $k_0$ such that for all $k\geq k_0$, $z_{k,i}=0$, so that the last inequality also holds true.
    
    We consider again the general case $z_\infty\geq 0$. Summing on $i$ we obtain that, almost surely, there exists $k_0$ such that, for all $k\geq k_0$ and all $j\leq q$,
    \[
    \sum_{\substack{i=1\\ \mathbb V_{i,j}\neq 0}}^pkz_{\infty,i}(\VV_{i,j}-\varepsilon)\leq \sum_{i=1}^p \sum_{m=1}^{z_{k,i}}V_{i,j}^{(m)} \leq \sum_{\substack{i=1\\ \mathbb V_{i,j}\neq 0}}^pkz_{\infty,i}(\VV_{i,j}+\varepsilon),
    \]
    where we used the fact that $V_{i,j}^{(m)}=0$ almost surely if $\VV_{i,j}=0$.
    
    Define the matrices $\VV_{+}^{\varepsilon}:= \left((\VV_{i,j}+\varepsilon)\mathds{1}_{\VV_{i,j}\neq 0}\right)_{1 \leq i\leq p, 1\leq j\leq q}$ and $\mathbb{V}_{-}^{\varepsilon}:= \left((\VV_{i,j}-\varepsilon)\mathds{1}_{\VV_{i,j}\neq 0}\right)_{1\leq i\leq p,1\leq j\leq q}$. 
    Since the function $\xi$ is superadditive, in particular it is non decreasing. Hence, we get
    
    \begin{equation}
    \label{eq1}
    \dfrac{\xi(kz_\infty\VV_-^{\varepsilon})}{k}\leq \dfrac{1}{k}\xi\left(\left(\displaystyle\sum\limits_{i=1}^p\displaystyle\sum\limits_{m=1}^{z_{k,i}}V_{i,j}^{(m)}\right)_{1\leq j\leq q}\right)\leq \dfrac{\xi(kz_\infty\VV_+^{\varepsilon})}{k}. 
    \end{equation}
    
    Let define
    \[\delta = \dfrac{\varepsilon}{\min\limits_{i,j/ \VV_{i,j}\neq 0} \VV_{i,j}},\]
     and note that $\delta <1$ thanks to our choice of $\varepsilon$.
    
    Assume first that $\mathfrak M(z_\infty)<+\infty$ and note that
    \begin{align*}
    \limsup\limits_{k\to+\infty} \dfrac{\xi(kz_\infty\VV_+^{\varepsilon})}{k} 
    &= \limsup\limits_{k\to+\infty} \dfrac{1}{k} \xi\left(\left(\sum\limits_{\substack{i=1\\ 
            \VV_{i,j}\neq 0}}^p kz_{\infty,i}\left(1+\dfrac{\varepsilon}{\VV_{i,j}}\right)\VV_{i,j}\right)_{1\leq j\leq q}\right)\\
    &\leq \lim\limits_{k\to+\infty} \dfrac{1}{k}\xi \left(\left(\sum\limits_{i=1}^p kz_{\infty,i}(1+\delta)\VV_{i,j}\right)_{1\leq j\leq q}\right)\\
    &=(1+\delta)\mathfrak M(z_\infty),
    \end{align*}
     and similarly, $\liminf\limits_{k\to+\infty}\dfrac{\xi(kz_\infty\VV_-^{\varepsilon})}{k}\geq (1-\delta)\mathfrak M(z_\infty).$
    
    Hence, taking $k\to +\infty$ in \eqref{eq1}, we obtain 
    $$(1-\delta)\mathfrak M(z_\infty) \leq\liminf_{k\to+\infty}\dfrac{1}{k}\xi\left(\left(\displaystyle\sum\limits_{i=1}^p\displaystyle\sum\limits_{m=1}^{z_{k,i}}V_{i,j}^{(m)}\right)_{j=1}^{q}\right) \leq (1+\delta)\mathfrak M(z_\infty)$$
    
    Finally, taking $\varepsilon\to 0$, then $\delta$ goes to $0$ and we conclude the desired result when $\mathfrak M(z_\infty)<+\infty$. If $\mathfrak M_{\ell}(z_\infty)=+\infty$ for some $\ell \in \{1,\dots, p\}$, the inequality $\liminf\limits_{k\to+\infty}\dfrac{\xi_{\ell}(kz_\infty\VV_-^{\varepsilon})}{k}\geq (1-\delta)\mathfrak M_\ell(z_\infty)$ still holds and so the result follows in this case.
    \end{proof}

    We now proceed with the proof of Theorem~\ref{thm:LGN1}.

    \begin{proof}[Proof of Theorem~\ref{thm:LGN1}]


 We first prove the almost sure convergence in Step~1,  and then the $L^1$ convergence in Step~2.

    \medskip\noindent\textit{Step 1. Almost sure convergence.}
    The result is trivial for $n=0$. By Lemma~\ref{lem: M(z)}, we have 
    \begin{equation}
    \label{eq:useful-cor2}
    \frac{Z^m_1}{m}=\frac{1}{m}\xi\left(\sum_{i=1}^p\sum_{k=1}^{z_{m,i}}V_{i,\cdot}^{(1,k)}\right)
    \xrightarrow[m\to\infty]{a.s.} \mathfrak M(z_\infty),
    \end{equation}
    If $\mathfrak M_\ell(z_\infty)> 0$ for some $\ell\in \{1,\ldots,p\}$, then this proves that $Z^m_{1,\ell}\sim_{m\to+\infty} m\, \mathfrak M_\ell(z_\infty)$ almost surely.  If $\mathfrak M_\ell(z_\infty) = 0$, then  $\xi_\ell(kz_\infty\VV)$ vanishes for all $k\geq 1$ and hence $Z_{1,\ell}\mathds 1_{Z_0\leq Cz_\infty}=0$ almost surely  for all $C>0$.
    Since $z_m\sim_{m\to+\infty} mz_\infty$, we deduce that there exists a (random) $m_0\geq 1$ such that, for all $m\geq m_0$, $Z^m_{1,\ell}=0$.
 Thus we proved that $Z^m_{1}\sim_{m\to+\infty} m\,\mathfrak M(z_\infty)$ almost surely, which proves the result when $n=1$.

    Assume now that $Z^m_n\sim_{m\to+\infty} m \mathfrak M^{n}(z_\infty)$ a.s. for some $n\geq 1$. Then the previous step with $z_m=Z^m_n$ entails that
    \[
    Z^m_{n+1}
    \sim_{m\to\infty} m\,\mathfrak M(\mathfrak M^{n}(z_\infty))=m\,\mathfrak M^{n+1}(z_\infty) \qquad \text{a.s.}
    \]
    This concludes the proof of the first assertion in Theorem~\ref{thm:LGN1}.
    
    \medskip\noindent\textit{Step 2. Convergence in $L^1$.}
    We prove now the $L^1$-convergence. 
    Denote $\onesq\in \NN^q$, $\mathbf 1_q=(1,\dots,1)$, and fix $z_0\in \NN^p$ such that $z_0\VV\geq \mathbf 1_q$. 
    Consider the bGWbp with initial position $Z_0^m=z_m$, $m\geq 1$, and denote by $W_1^m$ the number of children in the first generation. We have $W_1^m\leq |W_1^m|\mathbf 1_q \leq |W_1^m|z_0\VV$, and so using the second equality in~\eqref{eq:defM},
    \[
    Z_1^m=\xi(W_1^m)
    \leq \xi\left(|W_1^m|z_0\VV \right)
    \leq \mathfrak M\left(|W_1^m|z_0\right).
    \]
    Using Proposition~\ref{prop:post_hom_concave}, we deduce that
    \begin{align}
    \label{eq:Zisbounded}
    Z_1^m&\leq |W_1^m|\mathfrak M(z_0).
    \end{align}
    By assumption, the random vector 
    \[
    U^{(m)}:= \lfloor z_m/m \rfloor+1
    \] 
    is uniformly integrable,  and we have $z_m\leq m U^{(m)}$ almost surely, so that
    \begin{align}
    \label{eq:ineqMFU}
    0\leq |W_1^m|\leq \left|\sum_{i=1}^p \sum_{k=1}^{mU^{(m)}_i} V_{i,\cdot}^{(k,1)}\right|.
    \end{align}

    Since $U^{(m)}$ is independent from the other terms, we have
    \begin{multline}
    \label{eq:decomp1}
    \mathbb E\left(\left|\frac{1}{m}\sum_{j=1}^q\sum_{i=1}^p\sum_{k=1}^{mU^{(m)}_i} \left(V_{i,j}^{(k,1)}-\VV_{i,j}\right)\right|\right)
    \\
    \leq
     \sum_{u\in (\NN\setminus\{0\})^p} \sum_{j=1}^{q}\sum_{i=1}^p  \frac{1}{mu_i}\mathbb E\left(\left|\sum_{k=1}^{m u_i} \left(V_{i,j}^{(k,1)}-\VV_{i,j}\right)\right|\right)\,|u|\,\mathbb P\left(U^{(m)}=u\right).
    \end{multline}
    
    Using the law of large numbers, we deduce that, for each $i\in\{1,\cdots,p\}$ and $j\in\{1,\cdots,q\}$,
    \[
    \frac{1}{m}\EE\left(\left|\sum_{k=1}^{m} \left(V_{i,j}^{(k,1)}-\VV_{i,j}\right)\right|\right)\xrightarrow[m\to+\infty]{}0.
    \]
     In particular,  this ensures that the family 
    \[
    f_{m}:=\max_{u\in (\mathbb N\setminus\{0\})^p} \sum_{j=1}^{q}\sum_{i=1}^p\frac{1}{mu_i}\mathbb E\left(\left|\sum_{k=1}^{m u_i} \left(V_{i,j}^{(k,1)}-\VV_{i,j}\right)\right|\right)
    \] 
    converges to $0$ when $m\to+\infty$. In addition, for all $A>0$, 
    \begin{multline*}
    \sum_{u\in (\NN\setminus\{0\})^p} \sum_{j=1}^{q}\sum_{i=1}^p  \frac{1}{mu_i}\mathbb E\left(\left|\sum_{k=1}^{m u_i} \left(V_{i,j}^{(k,1)}-\VV_{i,j}\right)\right|\right)\,|u|\,\mathbb P\left(U^{(m)}=u\right)\\
    \begin{aligned}
     &\leq 
    \sum_{\substack{u\in (\mathbb N\setminus \{0\})^p\\ |u|\leq A}} f_m |u|\,\mathbb P\left(U^{(m)}=u\right)+\sum_{\substack{u\in (\NN\setminus \{0\})^p\\|u|>A}} \max_{n\in \NN} f_n |u|\,\mathbb P\left(U^{(m)}=u\right)\\
    &\leq f_m\,A+ \mathbb E\left(\left|U^{(m)}\right|\mathds{1}_{|U^{(m)}|>A}\right)\max_{n\in\NN} f_n .
    \end{aligned}
    \end{multline*}
    
    Since the family $(U^{(m)})_{m\geq 0}$ is uniformly integrable and choosing $A$ large enough, $\mathbb E\left(\left|U^{(m)}\right|\mathds{1}_{|U^{(m)}|>A}\right)$ can be made arbitrarily small uniformly in $m$, and, for any fixed $A$, choosing $m$ large enough, the term $f_m\,A$ can be chosen arbitrarily small. Using~\eqref{eq:decomp1}, this implies that
    \begin{align*}
    \mathbb E\left(\left|\frac{1}{m}\sum_{j=1}^q\sum_{i=1}^p\sum_{k=1}^{mU^{(m)}_i} \left(V_{i,j}^{(k,1)}-\VV_{i,j}\right)\right|\right)\xrightarrow[m\to+\infty]{} 0.
    \end{align*}
    
    In particular, this shows that $\frac{1}{m}\sum_{j=1}^q\sum_{i=1}^p\sum_{k=1}^{mU^{(m)}_i} V_{i,j}^{(k,1)}$ converges in $L^1$ and is thus uniformly integrable. By inequalities~\eqref{eq:ineqMFU} and~\eqref{eq:Zisbounded}, this entails that $(Z^m_1/m)_{m\geq 1}$ is uniformly integrable too. Now, since we also proved that $Z^m_1/m$ converges almost surely to $\mathfrak M(z_\infty)$, this implies that $Z^m_1/m$ converges in $L^1$ to $\mathfrak M(z_\infty)$.
    
    As above, the result for general $n\geq 1$ derives by iteration, which concludes the proof of Theorem~\ref{thm:LGN1}.
\end{proof}

We now turn to the proof of Corollary~\ref{cor:M_second_def}.
\begin{proof}[Proof of Corollary~\ref{cor:M_second_def}]
Theorem~\ref{thm:LGN1} yields Corollary~\ref{cor:M_second_def} when $\mathfrak M$ takes finite values. In the situation where  $\mathfrak M$ is not finite valued, we consider the vector $\onesp=(1,\dots, 1)\in \NN^p$ and introduce the superadditive function 
\[\hat{\xi}(x)=|x|\onesp.\]
Then, for  all $\alpha\in \NN$, we define the superadditive mating function
\[\xi_{(\alpha)} (x)=\min \{\xi(x),\alpha \hat{\xi}(x)\}:=\left(\min \{\xi(x)_i,\alpha \hat{\xi}(x)_i\}\right)_{1\leq i \leq p},\]
and we denote by $\mathfrak M_{(\alpha)}$ the function associated if we consider $\xi_{(\alpha)}$ as mating function with the same offspring distribution as the original process. We can check that $\mathfrak M_{(\alpha)}(z) = \min\{\mathfrak M(z),\alpha \hat{\xi}(z\VV)\}$ and so we obtain that $\mathfrak M_{(\alpha)}(z) \nearrow \mathfrak M(z)$ as $\alpha \to +\infty$, for all $z\in \RR^p_+$.
Since clearly $\xi_{(\alpha)}\nearrow \xi$, using the Monotone Convergence Theorem, for all $z\in\NN^p$ and $i\in\{1,\ldots,p\}$,
\begin{equation}
\EE_{(\alpha)}(Z_1|Z_0=z)\xrightarrow[\alpha\to+\infty]{}\EE(Z_1|Z_0=z), \label{eq:convergencebis}
\end{equation}
where $\EE_{(\alpha)}$ is the probability law associated to the process with mating function $\xi_{(\alpha)}$.
In particular, using Corollary~2 for the finite valued $\mathfrak{M}_{(\alpha)}$, for all $z\in\RR_+^p$,
\begin{align*}
\mathfrak M(z)&=\sup_{\alpha>0} \mathfrak M_{(\alpha)}(z)
=\sup_{\alpha>0} \sup_{m\geq 1} \frac{\EE_{(\alpha)}(Z_1|Z_0=\lfloor mz\rfloor )}{m}\\
              &= \sup_{m\geq 1} \sup_{\alpha>0}\frac{\EE_{(\alpha)}(Z_1|Z_0=\lfloor mz\rfloor )}{m}=\sup_{m\geq 1} \frac{\EE(Z_1|Z_0=\lfloor mz\rfloor )}{m}.
\end{align*}
Since $m\mapsto \EE(Z_1|Z_0=\lfloor mz\rfloor )$ defines a superadditive sequence, we deduce that
\begin{align*}
\mathfrak M(z)&=\lim_{m\geq 1} \frac{\EE(Z_1|Z_0=\lfloor mz\rfloor )}{m},
\end{align*}
which concludes the proof of Corollary~\ref{cor:M_second_def}.
\end{proof}

 Let us now prove Lemma~\ref{lem:domM}, hence also concluding the proof of Propositions~\ref{prop:primitivity} and~\ref{prop:anotherproponM}.

\begin{proof}[Proof of Lemma~\ref{lem:domM}]
    Let $z\in \NN^p$.  For $n=1$, we use Corollary~\ref{cor:M_second_def} and obtain that
    \[\mathfrak M(z)=\lim\limits_{k\to\infty}\dfrac{\EE(Z_1|Z_0=kz)}{k}=\sup\limits_{k\in \NN} \dfrac{\EE(Z_1|Z_0=kz)}{k}\geq \EE(Z_1|Z_0=z).\]
    
    Assume now that the inequality is true for some $n\in \NN$. Using the fact that $\mathfrak M$ is increasing (since it is superadditive), we obtain
	\begin{align*}
    \mathfrak M^{n+1}(z)&\geq \mathfrak M(\EE(Z_n|Z_0=z))\\
    &\geq \EE(\mathfrak M(Z_n)|Z_0=z)\\
    &\geq \EE\left(\EE(Z_1|Z_0=z')_{\vert{z'=Z_n}}|Z_0=z\right)\\
    &=\EE(Z_{n+1}|Z_0=z),
    \end{align*}    
    where in the second step we have used Jensen's inequality, since $\mathfrak M$ is concave by Proposition~\ref{prop:post_hom_concave}, and the last inequality is due to the Markov property. The proof is then complete.
\end{proof}

\section{Existence of the eigenelements and proof of Theorem~\ref{thm:ext}}
 \label{sec:eigenpbl}
 
    \subsection{The concave eigenvalue problem}
    \label{sec:krause}
    
    Consider $A$ a real strictly positive $N\times N$ matrix. A well-known result that goes back to Perron \cite{perron1907theorie} states that 
    \[\lim_{n\to \infty} \dfrac{A^nx}{\lambda^n}=c(x)v,\ \forall x\in \RR_+^p,\]
    where $\lambda$ is the greatest eigenvalue of $A$ with $v$ its corresponding eigenvector and $c$ is a suitable function. This result and its consequences are among the main tools used to study the asymptotic behaviour of the classical multi-type Galton-Watson process, applied to the expectation matrix associated with the process. In this section we give similar results: a theorem that goes back to Ulrich Krause \cite{krause1994relative} that provides us with the necessary tools to study the extinction conditions for the multi-type bGWbp.
    
    \begin{thm}[See \cite{krause1994relative} Section~4]
        \label{thm:eigenvalue} Consider $M:\RR_+^p\longrightarrow \RR_+^p$ a concave, primitive and positively homogeneous mapping. Then,
        
    \begin{enumerate}
        \item The eigenvalue problem $M(z)=\lambda z$ has a unique solution $(\lambda^*, z^*) \in \RR\times S^*$, with $\lambda^*>0$. If $(\lambda, x) \in \RR\times \left(\RR_+^p\setminus \{0\}\right)$ is another solution of the problem, then it must hold that $x=rz^*$ for some $r>0$ and $\lambda=\lambda^*$.
        \item The function $L:(\RR_+)^p\longrightarrow \RR_+ z^*$ given by $L(x)=\lim\limits_{k\to \infty} \frac{M^{k}(x)}{(\lambda^*)^k}$ exists on $\RR_+^p$ and holds $L(x)=\mathcal P(x)z^*$ where $\mathcal P:\RR_+^p \longrightarrow \RR_+$ is a concave and positively homogeneous mapping with $\mathcal P(x)>0$ for all $x\in \RR_+^p\setminus \{0\}$.
        \item $\lim\limits_{k\to \infty} \frac{M^{k}(x)}{|M^{k}(x)|}=z^*$ for all $x\in \RR_+^p\setminus\{0\}$.
        \item $\lim\limits_{k\to\infty} \frac{|M^{k+1}(x)|}{|M^{k}(x)|}=\lambda^*=\lim\limits_{k\to\infty}|M^{k}(x)|^{\frac{1}{k}}$ for all $x\in \RR_+^p\setminus \{0\}$.
        \item\label{item:thmKrausev} The convergence toward $L(x)=\mathcal P(x)z^*$ is uniform on $x\in S$.
    \end{enumerate}
        
    \end{thm}
    \subsection{Proof of Theorem \ref{thm:ext}}
\label{sec:proofMainTheorem}
If $\mathfrak{M}$ is finite, the existence of $\lambda^*$ and $z^*$ are guaranteed by Theorem~\ref{thm:eigenvalue}, Propositions~\ref{prop:post_hom_concave} and Proposition~\ref{prop:primitivity}. Note that, in the following of the proof, we make use of Theorem~\ref{thm:integrabilitycriterion} whose proof is developed in the next section and, except from the existence and uniqueness of $\lambda^*$ and $z^*$ which are yet established, does not use Theorem~\ref{thm:ext}.

If $\mathfrak M$ takes finite values and $\lambda^*\leq 1$, then,  by assertion $(2)$ in Theorem \ref{thm:eigenvalue},  for all $z\in \NN^p$, $(\mathfrak M^n(z))_{n\in \NN}$ is a bounded sequence.  
From Lemma~\ref{lem:domM}, $\mathfrak M^n(z) \geq \EE(Z_n|Z_0=z)$,  hence $\EE(Z_n|Z_0=z)$ is bounded for all $n\in \NN$ and so $Z_n$ does not converge to $+\infty$ with positive probability. The conclusion is then given by Assumption \ref{asu:bigasu}.\ref{asu:transience} since then $\lim\limits_{n\to \infty} |Z_n|$ can only be almost surely $0$, which finishes the proof of the theorem in the case $\lambda^*\leq 1$.
    
    If $\mathfrak M(z)<+\infty$ for all $z\in S$ and $\lambda^*>1$, Theorem~\ref{thm:integrabilitycriterion} entails that,  for all $\varepsilon \in (0,1)$, there exists $n_0\in \NN$ and $r>0$ such that, if $Z_0=z\in \NN^p$ with $|z|>r$, we have that with positive probability $Z_{n_0}\neq 0$ and $Z_n \geq (1-\varepsilon)^{n-n_0}\mathfrak M^{n-n_0}(Z_{n_0})$  for all $n\geq n_0$, with $(1-\varepsilon)^{n-n_0}\mathfrak M^{n-n_0}(Z_{n_0})\neq 0$ (since $\mathfrak M$ is primitive), and so we obtain $q_z<1$.    
    
 Assume now that there exist $z_0\in S$ and $i_0\in \{1,\dots ,p\}$ such that $(\mathfrak M(z_0))_{i_0}=+\infty$. 
 Consider, in the same way as for the proof of Corollary~\ref{cor:M_second_def}, the vector $\onesp=(1,\dots, 1)\in \NN^p$ and the function 
    \[\hat{\xi}(x)=|x|\onesp.\]
    
    For $\alpha\in \NN$ we define the function
    \[\xi_{(\alpha)} (x)=\min \{\xi(x),\alpha \hat{\xi}(x)\},\]
    which is superadditive, and we define $\mathfrak M_{(\alpha)}$ the function associated if we consider $\xi_{(\alpha)}$ as mating function with the same offspring distribution as the original process. We can check that $\mathfrak M_{(\alpha)}(z) = \min\{\mathfrak M(z),\alpha \hat{\xi}(z\VV)\}$ and so we obtain that $\mathfrak M_{(\alpha)}(z) \nearrow \mathfrak M(z)$ as $\alpha \to +\infty$, for all $z\in \RR^p_+$. Note that in particular 
    \begin{equation}
    \label{eq:divergence}
    \left(\mathfrak M_{(\alpha)}(z_0)\right)_{i_0}\xrightarrow[ ]{\alpha\to+\infty}+\infty.
    \end{equation}
    Since clearly $\xi_{(\alpha)}\nearrow \xi$, using the Monotone Convergence Theorem, for all $m\geq 1$ and $i\in\{1,\ldots,p\}$,
    \begin{equation}
    \EE_{(\alpha)}(Z_m|Z_0=ke_i)\xrightarrow[\alpha\to+\infty]{}\EE(Z_m|Z_0=ke_i). \label{eq:convergence}
    \end{equation}
    where $\EE_{(\alpha)}$ is the probability law associated to the process with mating function $\xi_{(\alpha)}$.
    
    By Assumption~\ref{asu:bigasu}.\ref{asu:primitive}, there exists $c_0>0,m\geq 1$ and $k\geq 1$ such that for all $i\in \{1,\dots,p\},$
    \[
    \EE(Z_m|Z_0=ke_i)\geq c_0 \onesp.
    \]
    By~\eqref{eq:convergence}, there exists $\alpha_0>0$ (which depends on $m$) such that for all $\alpha>\alpha_0$ and all $i\in \{1,\dots,p\}$,
    \[
    \EE_{(\alpha)}(Z_m|Z_0=ke_i)\geq \frac{c_0 \onesp}{2}\geq \dfrac{c_0}{2 \max\limits_{j\leq p} z_{0,j}} z_0.
    \]
    
    \noindent  This implies, by Lemma~\ref{lem:domM}, that for all $i\in \{1,\dots, p\}$,
    \[
    \mathfrak M_{(\alpha)}^{m}(e_i)\geq \dfrac{c_0}{2k \max\limits_{j\leq p} z_{0,j}}z_0.
    \]    
    Hence, by~\eqref{eq:divergence},
    \[
    \left(\mathfrak M^{m+1}_{(\alpha)}(e_i)\right)_{i_0} \geq \dfrac{c_0}{2k \max\limits_{j\leq p} z_{0,j}}(\mathfrak M_{(\alpha)}(z_0))_{i_0}\xrightarrow[ ]{\alpha\to +\infty} +\infty.
    \]
    This implies that
    \[
    \inf_{z\in S}\left(\mathfrak M^{m+1}_{(\alpha)}(z)\right)_{i_0}\xrightarrow[ ]{\alpha \to +\infty} +\infty.
    \]
    
    We remark that, since $\mathfrak M_{(\alpha)}$ is bounded over $S$, concave, positively homogeneous and primitive, there exists $\lambda_{\alpha}>0$ and $x_{\alpha}\in S$ such that $\mathfrak M_{(\alpha)}(x_{\alpha})=\lambda_{\alpha}x_{\alpha}$. Using this we have that 
    \[\lambda_{\alpha}^{m+1}=|\mathfrak M_{(\alpha)}^{m+1}(x_{\alpha})|\xrightarrow[ ]{\alpha\to +\infty} +\infty.\]
    
    We conclude that there exists $\alpha_0$ big enough such that $\lambda_{\alpha_0}>1$ and thanks to the previous computations the process with mating function $\xi_{\alpha_0}$ will not be almost surely extinct. Since $\xi_{(\alpha_0)}\leq \xi$, this process is stochastically dominated by the original process, and so we can find $r_{\alpha_0}>0$ such that for all $z\in \NN^p$ with $|z|>r_{\alpha_0}$, given $\{Z_0=z\}$, the original process has a positive probability of survival.
    
\section{Proof of Theorem~\ref{thm:integrabilitycriterion}}
    \label{sec:proof.thm.integrabilitycriterion}

In order to prove Theorem~\ref{thm:integrabilitycriterion}, we first prove that, if $\lambda^*>1$ and under the assumption that $\mathfrak M$ is bounded over $S$, we have that for all $\varepsilon \in (0,1)$, $\delta \in (0,\nicefrac1p]$,
\begin{equation}
\label{eq:first_limit}
\lim_{\substack{|z|\to+\infty\\ z\in U_\delta}} \PP \left(\forall n\in \NN,\,Z_{n+1} \in |(1-\varepsilon)\mathfrak M(Z_n),(1+\varepsilon)\mathfrak M(Z_n)|\right)=1,
\end{equation}
where $U_{\delta}$ is the set given by
\begin{equation}
\label{def:U_delta}
U_\delta=\{z\in \NN^p\,:\,z\geq |z|\delta \onesp\},
\end{equation}
where we recall that $\onesp=(1,\ldots,1)\in\NN^p$ and for all $a,b\in \RR_+^p$, $|a,b|:=\{z\in \RR_+^p,\ a\leq z\leq b\}$. We remark that, for $\delta>0$, $U_\delta$ is non-empty if and only if $\delta \in (0,\nicefrac1p]$.

Then, we prove that for any initial values,  either the process goes to extinction or it reaches a set $U_\delta$ in finite time.
Both results then lead to the proof of Theorem~\ref{thm:integrabilitycriterion}.

The second result is stated in Lemma~\ref{lem:minoration}, the first one is stated in Lemma~\ref{lem:capAn} and is based on Lemmas~\ref{lem:PAn}, \ref{lem:someconv} and \ref{lem:Miteratesgrowth} for which we introduction the following additional notation.

\medskip

For any $\varepsilon \in (0,1)$ and $n\geq 1$,  we consider the sequence of events 
\[A^{\varepsilon}_n:=\{\forall i\in \{1,\ldots,n\}, Z_i\in|(1-\varepsilon)\mathfrak M(Z_{i-1}),(1+\varepsilon)\mathfrak M(Z_{i-1})|\},
\] 
or simply $A_n$ when there is no risk of ambiguity. We also set $A^\varepsilon_0=\Omega$.

%
%
\begin{lem}
 \label{lem:PAn}
   Assume that $\mathfrak M$ is bounded on $S$. For any $\delta \in (0,\nicefrac1p]$ and $\varepsilon \in (0,1)$, there exists $c_0>0$  such that
 for all $z\in U_\delta$ and all $n\in\NN$,
   \begin{align*}
   \mathbb P\left(A_n^\varepsilon\mid Z_0=z\right)\geq 1-\sum_{i=1}^n c_0 \mathbb E\left(\mathds 1_{A^\varepsilon_{i-1}}f(|Z_{i-1}|)\mid Z_0=z\right).
   \end{align*}
    where
   \[
   f(x)=x\sum_{i=1}^p\sum_{j=1}^{q}\mathbb P(V_{i,j}>x)+\sum_{i=1}^p\sum_{j=1}^{q}\frac{\mathbb E(V_{i,j}^2\mathds 1_{V_{i,j}\leq x})}{x}.
   \]
    
\end{lem}

\begin{proof} We prove this lemma in two steps.\\
  \medskip\noindent \textit{Step 1.} We first consider the case $n=1$. That is, we prove that for $\delta \in (0,\nicefrac1p]$ and $\varepsilon \in (0,1)$, there exists $c_1$ such that, for $z\in U_\delta$,
  \begin{align}
  \label{eq:lemencadrement} 
    \mathbb{P}\left(Z_1\in |(1-\varepsilon)\mathfrak M(z),(1+\varepsilon)\mathfrak M(z)|\mid Z_0=z\right)\geq 1-
    c_1|z|\sum_{i=1}^p\sum_{j=1}^{q}\mathbb P(V_{i,j}>|z|)-c_1\sum_{i=1}^p\sum_{j=1}^{q}\frac{\mathbb E(V_{i,j}^2\mathds 1_{V_{i,j}\leq |z|})}{|z|}.
  \end{align}
   
   For $Z_0=z$, we have $Z_1=\xi(W_{1,1},\dots,W_{1,q})$ with 
   $
   W_{1,j}=\sum_{i=1}^p\sum_{k=1}^{z_i} V_{i,j}^{(k,1)}
   $
   for $1\leq j\leq q$.
   Fix $\delta_1\in(0,1)$ and $r_1>0$ (depending on $\delta_1$) such that, for all $z\in\mathbb N^p$ with $|z|\geq r_1$,
   \[
   (1-\delta_1) \mathbb V_{i,j} \leq \mathbb V_{i,j}^{\leq |z|}:=\mathbb E\left(V_{i,j}\mathds 1_{V_{i,j}\leq |z|}\right)\leq (1+\delta_1) \mathbb V_{i,j}.
   \]

For all   $|z|\geq r_1$ with $z\in U_\delta$, we have
   \begin{align*}
   \mathbb P((W_{1,1},\dots, W_{1,q})\geq  (1-\delta_1)^2z\VV\mid Z_0=z)
   &\geq \mathbb P(( W_{1,1},\dots, W_{1,q})\geq  (1-\delta_1)z\VV^{\leq|z|}\mid Z_0=z)\\
   &\geq 1-\sum_{j=1}^{q} \mathbb P( W_{1,j}< (1-\delta_1)\sum_{i=1}^pz_i\VV_{i,j}^{\leq|z|}\mid Z_0=z)
   \end{align*}
   but for $j\in \{1,\ldots, q\}$,
   \begin{align*}
   \mathbb P( W_{1,j}< (1-\delta_1)\sum_{i=1}^pz_i\VV_{i,j}^{\leq|z|}\mid Z_0=z)
   &\leq \mathbb P(\exists i\in\{1,\ldots,p\},k\in\{1,\ldots,z_i\}\text{ s.t }V_{i,j}^{(k,1)}>|z|\mid Z_0=z)\\
   &\quad+\mathbb P\left(\sum_{i=1}^p\sum_{k=1}^{z_i} \left(V_{i,j}^{(k,1)}\mathds 1_{ V_{i,j}^{(k,1)}\leq |z|}-\VV_{i,j}^{\leq |z|}\right)< -\delta_1\sum_{i=1}^pz_i\VV_{i,j}^{\leq |z|}\mid Z_0=z\right)\\
   &\leq \sum_{i=1}^p z_i\, \mathbb P(V_{i,j}>|z|)+\frac{\text{Var}\left(\sum_{i=1}^p\sum_{k=1}^{z_i} V_{i,j}^{(k,1)}\mathds 1_{ V_{i,j}^{(k,1)}\leq |z|}\mid Z_0=z\right)}{\delta_1^2\left(\sum_{i=1}^pz_i\VV_{i,j}^{\leq |z|}\right)^2}\\
   &\leq |z| \sum_{i=1}^p \mathbb P(V_{i,j}>|z|)+\frac{\sum\limits_{i=1}^p\sum\limits_{k=1}^{z_i}\text{Var}\left( V_{i,j}^{(k,1)}\mathds 1_{ V_{i,j}^{(k,1)}\leq |z|}\mid Z_0=z\right)}{\delta_1^2\delta^2(1-\delta_1)^2\left(\sum\limits_{i=1}^p\VV_{i,j}\right)^2|z|^2}\\
   &\leq |z| \sum_{i=1}^p \mathbb P(V_{i,j}>|z|)+\frac{\sum\limits_{i=1}^p \EE\left( V_{i,j}^2\mathds{1}_{V_{i,j}\leq |z|}\right)}{\delta_1^2\delta^2(1-\delta_1)^2\left(\sum\limits_{i=1}^p\VV_{i,j}\right)^2|z|},
   \end{align*}
   where we used the independence of the random variables $V_{i,j}^{(k,1)}$, the fact that $z_i\geq \delta|z|$ for all $i\in\{1,\ldots,p\}$, and $\VV_{i,j}^{\leq |z|}\geq (1-\delta_1) \VV_{i,j}$. Proceeding similarly for the event $\{( W_{1,1},\dots, W_{1,q})\leq  (1+\delta_1)^2z\VV\}$, we deduce that there exists a constant $c>0$ such that
   \begin{multline}
   \label{eq:step1}
   \mathbb P( W_1\in |(1-\delta_1)^2z\VV,(1+\delta_1)^2z\VV|\mid Z_0=z)\\
   \geq 1-
   c|z|\sum_{i=1}^p \sum_{j=1}^q\mathbb P(V_{i,j}>|z|)-c\sum_{i=1}^p \sum_{j=1}^q\frac{\mathbb E(V_{i,j}^2\mathds 1_{V_{i,j}\leq |z|})}{|z|},
   \end{multline}
   where $ W_1=( W_{1,1},\ldots, W_{1,q})$.

   If we now apply Proposition~\ref{prop:M} with the compact set $U_\delta \cap S$, for all $\varepsilon'>0$, there exists $r_2>0$ such that if $\ell > r_2(1-\delta_1)^2$,
   \[
   \left|\frac{\xi(\ell u \VV)}{\ell}-\mathfrak M(u)\right|\leq \varepsilon',\ \forall u\in U_\delta\cap S,
   \]
   and we deduce that for all $z\in U_\delta$ such that $|z|\geq r_2$,
   \begin{align*}
   \left|\frac{\xi((1-\delta_1)^2z\VV)}{(1-\delta_1)^2|z|}-\frac{\mathfrak M(z)}{|z|}\right|\leq \varepsilon'.
   \end{align*}
   
   Hence, for all $z\in U_\delta$ with $|z|\geq r_2$,
   \begin{align*}
   \frac{\xi((1-\delta_1)^2z\VV)}{(1-\delta_1)^2|z|}\geq \mathfrak M(z/|z|)-\varepsilon'.
   \end{align*}
   In addition, since $\mathfrak M$ is concave on $S$, it is locally Lipschitz on $S^*$ and in particular each of its components are uniformly bounded away from $0$ on $U_\delta \cap S\subset S^*$ by a constant $m_1>0$ (which  depends on $\delta \in (0,\nicefrac1p]$). Hence, for all $z\in U_\delta$ with $|z|\geq r_2$, 
   \begin{align*}
   \frac{\xi((1-\delta_1)^2z\VV)}{(1-\delta_1)^2|z|}\geq \mathfrak M(z/|z|)(1-\varepsilon'/m_1).
   \end{align*}
   Similarly, there exist $r_3>0$ and $m_2>0$ such that, for all $z\in U_\delta$ with $|z|\geq r_3$,
   \begin{align*}
   \frac{\xi((1+\delta_1)^2z\VV)}{(1+\delta_1)^2|z|}\leq \mathfrak M(z/|z|)(1+\varepsilon'/m_2).
   \end{align*}
   Hence, for $z\in U_\delta$ with $|z|\geq r_1\vee r_2 \vee r_3$,
   \begin{multline}
   \{( W_{1,1},\dots, W_{1,q})\in |(1-\delta_1)^2z\VV,(1+\delta_1)^2z\VV|\}\\
   \subset \{Z_1\in|\mathfrak M(z)(1-\varepsilon'/m_1)(1-\delta_1)^2, \mathfrak M(z)(1+\varepsilon'/m_2)(1+\delta_1)^2|\}.
   \end{multline}
   Choosing $\varepsilon'$ and $\delta_1$ small enough, we deduce that there exists $r_4>0$ such that, for all $z\in U_\delta$ with $|z|\geq r_4$,~\eqref{eq:lemencadrement} holds true for some constant $c_1$. Up to a change in the constant $c_1$, we deduce that this is true for all $z\in U_\delta$.
   
    \medskip\noindent \textit{Step 2.} We iterate now the result obtained in the previous step. We have, for all $z\in U_\delta$.
    \begin{align*}
     \mathbb P\left(Z_1\in|(1-\varepsilon)\mathfrak M(z),(1+\varepsilon)\mathfrak M(z)|\mid Z_0=z\right)
    &\geq 1-c_1f(|z|)
    \end{align*}
    for some constant $c_1>0$. Then, observing that $Z_1\in|(1-\varepsilon)\mathfrak M(z),(1+\varepsilon)\mathfrak M(z)|$ implies that $Z_1\in ||z|(1-\varepsilon)\mathfrak M(z/|z|),(1+\varepsilon)|z|\mathfrak M(z/|z|)|$ and hence that
    \begin{align*}
    Z_1/|Z_1|\geq \frac{(1-\varepsilon)\mathfrak M(z/|z|)}{(1+\varepsilon)|\mathfrak M(z/|z|)|}\geq \delta'_1\onesp,
    \end{align*}
    with
    \begin{align*}
    \delta'_1:=\frac{(1-\varepsilon)\min_{u\in U_\delta \cap S,\,i\in\{1,\ldots,p\}} \langle \mathfrak M(u),e_i\rangle}{(1+\varepsilon)\sup\limits_{u\in S}|\mathfrak M(u)|}.
    \end{align*}
    where $\min_{u\in U_\delta\cap S,\,i\in\{1,\ldots,p\}} \langle \mathfrak M(u),e_i\rangle>0$ by Proposition~\ref{prop:anotherproponM}. Now, applying the same reasoning as in \textit{Step 1} but with $\delta_1$ instead of $\delta$, we deduce that there exists a constant $c_2>0$ such that, on the event $\{Z_0=z\}$,
    \begin{align*}
    \mathbb P\left(Z_2\in|(1-\varepsilon)\mathfrak M(Z_1),(1+\varepsilon)\mathfrak M(Z_1)|\mid Z_1\right)
    &\geq \mathds{1}_{Z_1\in|(1-\varepsilon)\mathfrak M(z),(1+\varepsilon)\mathfrak M(z)|} (1-c_2f(|Z_1|))\\
    &=\mathds 1_{A^\varepsilon_1}(1-c_2f(|Z_1|)).
    \end{align*}
    And hence, using Markov's property at time $1$,
    \begin{align*}
    \mathbb P\left(A^\varepsilon_2\mid Z_0=z\right)
    &\geq \mathbb E\left(\mathbb P(Z_2\in|(1-\varepsilon)\mathfrak M(Z_1),(1+\varepsilon)\mathfrak M(Z_1)|\mid Z_1)\mathds{1}_{A^\varepsilon_1}\mid Z_0=z\right)\\
    &\geq \mathbb P(A^\varepsilon_1\mid Z_0=z)-\mathbb{E}\left(\mathds{1}_{A^\varepsilon_1} c_2f(|Z_1|)\mid Z_0=z\right)\\
    &\geq 1-c_1f(|z|)-c_2\mathbb{E}\left(\mathds{1}_{A^\varepsilon_1} f(|Z_1|)\mid Z_0=z\right)
    \end{align*}

    Iterating this procedure, we deduce that there exists a positive sequence $(c_n)_{n\geq 1}$ such that, for all $n\geq 1$,
    \begin{align*}
    \mathbb P\left(A^\varepsilon_n\mid Z_0=z\right)
    &\geq 1 -\sum_{i=1}^n c_i\mathbb E\left(\mathds 1_{A^\varepsilon_{i-1}}f(|Z_{i-1}|)\mid Z_0=z\right).
    \end{align*}    
    According to Theorem~\ref{thm:eigenvalue}~(\ref{item:thmKrausev}), there exists $n_0$ such that
    \begin{align*}
    \sup_{u\in S} \left|(\lambda^*)^{-{n_0}}\mathfrak M^{n_0}(u)-\mathcal P(u)z^*\right|\leq \frac{1}{2}\inf_S \mathcal P\,\min_{i\in\{1,\ldots,p\}} z^*_i,
    \end{align*}
    so that $\mathfrak M^{n_0}(u)\in\left|(\lambda^*)^{{n_0}}\mathcal P(u)z^*/2,3(\lambda^*)^{{n_0}}\mathcal P(u)z^*/2\right|$, $\forall u\in S$.
    Then, under $A^\varepsilon_{n_0}$, we have 
    \[
    Z_{n_0}\in|(1-\varepsilon)^{n_0}\mathfrak M^{n_0}(z),(1+\varepsilon)^{n_0}\mathfrak M^{n_0}(z)|\subset \left|(1-\varepsilon)^{n_0}(\lambda^*)^{n_0}\mathcal P(z)z^*/2,3(1+\varepsilon)^{n_0}(\lambda^*)^{n_0}\mathcal P(z)z^*/2\right|,
    \]
    so that (recall that $|z^*|=1$)
    \begin{align*}
    Z_{n_0}/|Z_{n_0}|\geq \frac{(1-\varepsilon)^{n_0}(\lambda^*)^{n_0}\mathcal P(z)z^*/2}{3(1+\varepsilon)^{n_0}(\lambda^*)^{n_0}\mathcal P(z)/2}
    \geq \min z^*\,\frac{(1-\varepsilon)^{n_0}}{3(1+\varepsilon)^{n_0}}\,\onesp.
    \end{align*}
    Since $\delta':=\min z^*\,\frac{(1-\varepsilon)^{n_0}}{3(1+\varepsilon)^{n_0}}$ does not depend on $\delta$, we deduce from \textit{Step 1} that there exists $c'_1$ which does not depend on $\delta$ such that
    \begin{align*}
    \mathbb P(A^\varepsilon_{n_0+1}\mid Z_0,\ldots,Z_{n_0})\geq \mathds 1_{A^\varepsilon_{n_0}}\left(1-c'_1f(|Z_{n_0}|)\right).
    \end{align*}
    Iterating the procedure of the beginning of the proof, we deduce that there exist $c'_2,\ldots,c'_{n_0}$ which do not depend on $\delta>0$ such that, for all $n\in\{n_0+1,\ldots,2n_0\}$,
        \begin{align*}
    \mathbb P\left(A^\varepsilon_n\mid Z_0=z\right)
    &\geq 1-\sum_{i=1}^{n_0} c_i\mathbb E\left(\mathds 1_{A^\varepsilon_{i-1}}f(|Z_{i-1}|)\mid Z_0=z\right)-\sum_{i=n_0+1}^{n} c'_{i-n_0}\mathbb E\left(\mathds 1_{A^\varepsilon_{i-1}}f(|Z_{i-1}|)\mid Z_0=z\right)
    \end{align*}    
    Under $A^\varepsilon_{2n_0}$, we have 
    \begin{align*}
    Z_{2n_0}\in |(1-\varepsilon)^{n_0}\mathfrak M^{n_0}(Z_{n_0}),(1+\varepsilon)^{n_0}\mathfrak M^{n_0}(Z_{n_0})|\subset \left|(1-\varepsilon)^{n_0}(\lambda^*)^{n_0}\mathcal P(Z_{n_0})z^*/2,3(1+\varepsilon)^{n_0}(\lambda^*)^{n_0}\mathcal P(Z_{n_0})z^*/2\right|,
    \end{align*}
    hence, using the same computations as above, we have 
    \begin{align*}
    \mathbb P(A^\varepsilon_{2n_0+1}\mid Z_0,\ldots,Z_{2n_0})\geq \mathds 1_{A^\varepsilon_{2n_0}}\left(1-c'_1f(|Z_{2n_0}|)\right),
    \end{align*}
    with the same constant $c'_1$. Iterating the procedure of the beginning of the proof, we deduce that, for all $n\in \{2n_0+1,\ldots,3n_0\}$,
    \begin{multline*}
    \mathbb P\left(A^\varepsilon_n\mid Z_0=z\right)
    \geq 1-\sum_{i=1}^{n_0} c_i\mathbb E\left(\mathds 1_{A^\varepsilon_{i-1}}f(|Z_{i-1}|)\mid Z_0=z\right)\\-\sum_{i=n_0+1}^{2n_0} c'_{i-n_0}\mathbb E\left(\mathds 1_{A^\varepsilon_{i-1}}f(|Z_{i-1}|)\mid Z_0=z\right)\\-\sum_{i=2n_0+1}^{n} c'_{i-2n_0}\mathbb E\left(\mathds 1_{A^\varepsilon_{i-1}}f(|Z_{i-1}|)\mid Z_0=z\right)
    \end{multline*}    
    Proceeding by induction and taking $c_0=\max_{i\in\{1,\ldots,n_0\}}c_i\vee c'_i$, this concludes the proof of the lemma.
\end{proof}

We prove now a useful auxiliary lemma.

\begin{lem}
    \label{lem:someconv}
    Let $(x_n)_{n\geq 0}$ be a positive sequence such that, for some $\alpha >1$ and $c>0$,
    \begin{align}
    \label{eq:oneassumption}
    x_n\geq c\alpha^{n-k} x_k,\quad\forall n\geq k\geq 0.
    \end{align}
    Then
    \begin{align*}
    \sum_{n\geq 0}  f(x_n)\leq \frac{2\alpha}{c(\alpha-1)}\sum_{i=1}^p\sum_{j=1}^{q}\mathbb E\left(\frac{V_{i,j}^2}{(cx_0)\vee V_{i,j}}\right)
    \end{align*}
    with $f$ defined in Lemma~\ref{lem:PAn}.
\end{lem}

\begin{proof}
     We first consider the first part and then the second part of $f(x_n)$, for each $i\in\{1,\ldots,p\}$ and $j\in\{1,\ldots,q\}$.
    
    We have, using Fubini's Theorem,
    \begin{align*}
    \sum_{n\geq 0} x_n\mathbb P(V_{i,j}>x_n)
    &=\mathbb E\left(\sum_{n\geq 0} x_n\mathds 1_{V_{i,j}>x_n}\right)\\
    &\leq \mathbb E\left(\mathds 1_{V_{i,j}>\min\limits_{n\in\NN} x_n}\sum_{n=0}^{N_{ij}} x_n\right),
    \end{align*}
    where $N_{ij}:=\max\{n\geq 0, x_n< V_{i,j}\}$, with the convention that $\max \emptyset=-1$ (note that, since $\alpha>1$, $x_n\to+\infty$ so that $N_{ij}<+\infty$). Inequality~\eqref{eq:oneassumption} entails that $\min_{n\in\NN} x_n\geq cx_0$ and that, for all $n\leq N_{ij}$, $x_n\leq \alpha^{n-N_{ij}}x_{N_{ij}}/c\leq \alpha^{n-N_{ij}}V_{i,j}/c$. We deduce that
    \begin{align*}
    \sum_{n\geq 0} x_n\mathbb P(V_{i,j}>x_n)
    &\leq \mathbb E\left(\mathds 1_{V_{i,j}>c x_0} V_{i,j} \sum_{n=0}^{N_{ij}} \alpha^{n-N_{ij}}/c\right)\\
    &\leq \mathbb E\left(\mathds 1_{V_{i,j}>c x_0} V_{i,j} \frac{\alpha}{c(\alpha-1)}\right)\\
    &\leq  \frac{\alpha}{c(\alpha-1)}\mathbb E\left(\frac{V_{i,j}^2}{(cx_0)\vee V_{i,j}}\right),
    \end{align*}
    where we used the fact that $\mathds 1_{V_{i,j}>c x_0}\leq \frac{V_{i,j}}{(cx_0)\vee V_{ij}}$ almost surely.
    
    Using again Fubini's Theorem, we have
    \[
    \sum_{n\geq 0} \frac{\mathbb E\left(V_{i,j}^2\mathds 1_{V_{i,j}\leq x_n}\right)}{x_n}
    =\mathbb E\left(V_{i,j}^2  \sum_{n\geq 0}\frac{\mathds 1_{V_{i,j}\leq x_n}}{x_n}\right)
    \leq \mathbb E\left(V_{i,j}^2  \sum_{n=N'_{ij}}^{+\infty}\frac{1}{x_n}\right),
    \]
    where $N'_{ij}=\min\{n\geq 0, x_n\geq V_{i,j}\}=N_{ij}+1$.  Inequality~\eqref{eq:oneassumption} entails that, for all $n\geq N'_{ij}$, $1/x_n\leq \alpha^{N'_{ij}-n}/(cx_{N'_{ij}})$. Hence, using the fact that $x_{N'_{ij}}\geq V_{i,j}$ by definition of $N'_{ij}$, we obtain
    \[
    \sum_{n\geq 0} \frac{\mathbb E\left(V_{i,j}^2\mathds 1_{V_{i,j}\leq x_n}\right)}{x_n}
    \leq \mathbb E\left(V_{i,j}^2\frac{1}{x_{N'_{ij}}\vee V_{i,j}} \sum_{n=N'_{ij}}^{+\infty}\frac{\alpha^{N'_{ij}-n}}{c}\right)
    \leq \mathbb E\left(\frac{V_{i,j}^2}{(cx_0)\vee V_{i,j}}\frac{\alpha}{c(\alpha-1)}\right),
    \]
    where we used the fact that $x_{N'_{ij}}\geq cx_0$ by~\eqref{eq:oneassumption}.
    
    Summing over $i\in\{1,\ldots,p\}$ and $j\in\{1,\ldots,q\}$, this concludes the proof of Lemma~\ref{lem:someconv}.
\end{proof}

Now we state a second auxiliary lemma, where we prove that for all $z\in \NN^p$, the sequence $\left(\mathfrak M^n(z)\right)_{n\in\NN}$ holds the property~\eqref{eq:oneassumption}.
\begin{lem}
    \label{lem:Miteratesgrowth}
     Assume that $\mathfrak M$ is bounded on $S$ and $\lambda^*>1$. There exists a constant $c_0\in (0,1]$ and $\lambda\in(1,\lambda^*)$ such that, for all $z\in \NN^p$ and all $n\geq 1$ and $k\in\{0,\ldots,n\}$,
     \begin{align*}
     |\mathfrak M^n(z)|\geq c_0 \lambda^{n-k}  |\mathfrak M^{k}(z)|.
     \end{align*}
\end{lem}

\begin{proof}
     Let $\delta\in(0,1)$ such that $\lambda:=\frac{1-\delta}{1+\delta}\lambda^*>1$. Fix $z\in \NN^p\setminus\{0\}$. If there exists $x>0$ such that $z\in|(1-\delta)xz^*,(1+\delta)x z^*|$, then $|z|\leq |(1+\delta)x z^*|=(1+\delta)x $, and, since $M$ is increasing and positively homogeneous,
     \begin{align}
     \label{eq:Mincr}
     |\mathfrak M(z)|\geq |\mathfrak M((1-\delta)xz^*)|=|(1-\delta)x\mathfrak M(z^*)|=|(1-\delta)x\lambda^* z^*|= (1-\delta)x\lambda^*\geq \lambda|z|.
     \end{align}
     Moreover,  Theorem~\ref{thm:eigenvalue}~(\ref{item:thmKrausev}) 
     entails that there exists $n_0\geq 1$ such that, for any $z\in S$ and $n\geq n_0$,
     \[
     \mathfrak M^n(z)\in |(1-\delta)(\lambda^*)^{n} C(z) z^*,(1+\delta)(\lambda^*)^{n} C(z) z^*|,
     \]
     so that, according to~\eqref{eq:Mincr}, for $n\geq k \geq n_0$ 
     \[
     |\mathfrak M^{n+1}(z)|\geq\lambda|\mathfrak M^n(z)|\geq \lambda^{n+1-k} |\mathfrak M^{k}(z)|\geq \lambda^{n+1-n_0} |\mathfrak M^{n_0}(z)|,
     \]
     By homogeneity of $\mathfrak M$, this extends to all $z\in \NN^p$. 

     For all $n\in\{0,\ldots,n_0\}$ and all $z\in (\NN\setminus\{0\})^p$, we have, for all $k\in\{0,\ldots,n\}$,
     \begin{align*}
     |\mathfrak M^{n+1}(z)|&\geq  \left|\mathfrak M\left(\frac{\mathfrak M^n(z)}{|\mathfrak M^n(z)|}\right)\right|\,|\mathfrak M^n(z)|\geq \inf_S |\mathfrak M|\,|\mathfrak M^n(z)|\geq (\inf_S |\mathfrak M|)^{n+1-k}\,|\mathfrak M^{k}(z)|\\
                   &\geq c_1^{n+1-k} \lambda^{n+1-k} |\mathfrak M^{k}(z)|,
     \end{align*}
    where $c_1= \lambda^{-1}\inf_S |\mathfrak M| $. Setting $c_0=1\wedge c_1^{n_0}$ concludes the proof of Lemma~\ref{lem:Miteratesgrowth}.
\end{proof}

We are now in position to compute the limit~\eqref{eq:first_limit}.

    \begin{lem}
        \label{lem:capAn}
        Assume that $\mathfrak M$ is bounded on $S$ and $\lambda^*>1$. For any $\delta\in (0,\nicefrac1p]$ and $\varepsilon \in (0,1)$, we have
        \begin{align*}
        \mathbb P\left(\bigcap_{n\geq 1} A_n^\varepsilon\mid Z_0=z\right)\xrightarrow[|z|\to+\infty,\,z\in U_\delta]{} 1.
        \end{align*}
    \end{lem}

\begin{proof}
 Take $c_0>0$ and $\lambda\in(1,\lambda^*)$ from Lemma~\ref{lem:Miteratesgrowth}.   We assume without loss of generality that $\alpha:=(1-\varepsilon)\lambda>1$.
 For all $i\geq k\geq 1$, on the event $A_i$, we have then
 \[
 |Z_i|\geq |(1-\varepsilon)^k\mathfrak M^k(Z_{i-k})|\geq  c_0 \lambda^k (1-\varepsilon)^k |Z_{i-k}|= c_0 \alpha^k |Z_{i-k}|.
 \]
Hence, according to Lemma~\ref{lem:someconv}, almost surely, on the event $\{Z_0=z\}$,
\[
\sum_{n\geq 1}\mathds 1_{A_{n-1}}f(|Z_{n-1}|)\leq \frac{2\alpha}{c_0(\alpha-1)}\sum_{i=1}^p\sum_{j=1}^{q}\mathbb E\left(\frac{V_{i,j}^2}{(c_0|z|)\vee V_{i,j}}\right).
\]
We deduce that
\begin{align*}
\sum_{n\geq 1} \mathbb E\left(\mathds 1_{A_{n-1}}f(|Z_{n-1}|)\mid Z_0=z\right)&\leq \frac{2\alpha}{c_0(\alpha-1)}\sum_{i=1}^p\sum_{j=1}^{q}\mathbb E\left(\frac{V_{i,j}^2}{(c_0|z|)\vee V_{i,j}}\right).
\end{align*}
Letting $n\to+\infty$ in Lemma~\ref{lem:PAn}, we  obtain that there exists $c'_0>0$ such that, for all $z\geq \delta|z|\onesp$ (recall that $A_{n-1}\subset A_n$ for all $n\geq 1$), 
\begin{align*}
\mathbb P\left(\bigcap_{n\geq 1} A_n\mid Z_0=z\right)&=\lim_{n\to+\infty} \mathbb P\left(\bigcap_{k= 1}^n A_n\mid Z_0=z\right)\\
&\geq 1-c'_0\sum_{n\geq 1} \mathbb E\left(\mathds 1_{A_{n-1}}f(|Z_{n-1}|)\mid Z_0=z\right)\\
&\geq 1-\frac{2c'_0\alpha}{c_0(\alpha-1)}\sum_{i,j}\mathbb E\left(\frac{V_{i,j}^2}{(c_0|z|)\vee V_{i,j}}\right).
\end{align*}
But $V_{i,j}$ is integrable for all $i,j$ and hence, by dominated convergence theorem,
\begin{align*}
\mathbb E\left(\frac{V_{i,j}^2}{(c_0|z|)\vee V_{i,j}}\right)\xrightarrow[|z|\to+\infty]{}0.
\end{align*}
This concludes the proof of Lemma~\ref{lem:capAn}.
\end{proof}

Lemma~\ref{lem:capAn} is useful for large starting values $z$ such that $z\geq \delta|z|\onesp$. In order to use it for all initial positions, we show that such values are eventually reached by the process in finite time.
\begin{lem}
    \label{lem:minoration}
    Assume that $\mathfrak M$ is bounded on $S$. There exist $\delta_0\in (0,\nicefrac1p]$ and $n_0\geq 1$ such that, for all $r>0$,
    for all $\eta \in (0,1)$, there exists $\rho>0$ such that
    \[
    \inf_{z\in \mathbb N^p,\ |z|\geq \rho} \mathbb P\left(\tau_{\delta_0,r}\leq n_0\mid Z_0=z\right)\geq 1-\eta,
    \]
    where
    \[
    \tau_{\delta_0,r}=\inf\{n\geq 0,\ Z_n\geq |Z_n|\delta_0\onesp\text{ and }|Z_n|\geq r\}.
    \] 
    In addition,
    \[
    \inf_{z\in \mathbb N^p} \mathbb P\left(\tau_{\delta_0,r}\wedge T_0<+\infty\mid Z_0=z\right)=1,
    \]
    with $T_0:=\inf \{n\geq 0, \, |Z_n|=0 \}$ the extinction time of the process.
\end{lem}

\begin{proof}
Let $n_0\geq 1$ and $\delta_1>0$ such that, for all $i\in\{1,\ldots,p\}$, $\mathfrak M^{n_0}(e_i)>\delta_1\onesp$. Such $n_0$ and $\delta_1$ exist as $\mathfrak M$ is primitive by Proposition~\ref{prop:primitivity}.
In order to prove the first inequality of the lemma, we show that there exists $\alpha>0$ such that, for all $\eta \in (0,1)$, there exists $k_1>0$ such that
\begin{align}
    \label{eq:boundzn0overz}
    \mathbb P\left(Z_{n_0}\geq |z| \delta_2\onesp\mid Z_0=z\right)\geq 1-\eta/2\,, \qquad \forall |z|\geq k_1
\end{align}
with $\delta_2:=\nicefrac{\delta_1}{2p}$, and there exists $k_2>0$ such that
\begin{align}
    \label{eq:boundzn0overz2}
    \mathbb P\left(\alpha\,|z|>|Z_{n_0}|\mid Z_0=z\right)\geq 1- \eta/2\, \qquad \forall |z|\geq k_2.
 \end{align}

Once this is proved, we set  $\delta_0=\delta_2/\alpha$ (which does not depend on $\eta$). Then for $r>0$, setting $\rho = k_1\vee k_2 \vee \nicefrac r{\delta_2 p}$, 
we obtain from the general property $\mathbb P(A\cap B)\geq \mathbb P(A)+\mathbb P(B)-1$ that, for all $z$ such that $|z|\geq \rho$,
\begin{align*}
    \mathbb P\left(\tau_{\delta_0,r}\leq n_0\mid Z_0=z\right)
  & \geq \mathbb P \left(Z_{n_0}\geq |Z_{n_0}|\delta_0\onesp\text{ and }|Z_{n_0}|\geq r \text{ and } |Z_{n_0}| < \alpha\,|z| \mid Z_0=z\right)
 \\
  & \geq \mathbb P \left(Z_{n_0}\geq \alpha |z|\,\delta_0\onesp \text{ and } |Z_{n_0}| < \alpha\,|z| \mid Z_0=z\right)
  \\
  &\geq \mathbb P \left(Z_{n_0}\geq  |z|\delta_2\onesp \mid Z_0=z\right)  +\mathbb P \left(|Z_{n_0}| < \alpha\,|z| \mid Z_0=z\right)-1\\
  &\geq 1 - \eta,
    \end{align*}
    where we used for the second inequality that $|\alpha|z|\delta_0\onesp|\geq r$ for all $|z|\geq \rho$.
   
Let us prove \eqref{eq:boundzn0overz} and \eqref{eq:boundzn0overz2}. By definition of $n_0\geq 1$ and $\delta_1>0$ and by Theorem~\ref{thm:LGN1}, we deduce that there exists $k_0\geq 1$ such that, for all $k\geq k_0$ and all $i\in\{1,\ldots,p\}$, 
    \begin{align*}
    \mathbb P(Z_{n_0}\geq k\delta_1\onesp/2\mid Z_0=ke_i)\geq (1-\eta/2)^{\nicefrac1p}.
    \end{align*}

     Using the auxiliary lemma in Appendix~\ref{sec:annex1}, we obtain
    \begin{align*}
    \mathbb P\left(Z_{n_0}\geq \sum_{i=1}^p \mathds{1}_{z_i\geq k_0}z_i \delta_1\onesp/2\mid Z_0=z\right)&\geq\prod_{i=1}^p \mathbb P\left(Z_{n_0} \geq \mathds{1}_{z_i\geq k_0} z_i \delta_1\onesp/2 \mid Z_0=z_ie_i\right)\\
    &\geq \prod_{i=1}^p ( 1-\eta/2)^{\nicefrac1p}=1-\eta/2.
    \end{align*}

     If in addition $|z|\geq k_1:=pk_0$, then 
    \[
    \frac{\delta_1}{2}\sum_{i=1}^p \mathds{1}_{z_i\geq k_0}z_i\geq \frac{\delta_1}{2}(|z|-(p-1)k_0)
    =\frac{\delta_1}{2}|z|\left(1-\frac{(p-1)k_0}{|z|}\right)\geq \frac{\delta_1}{2\,p} |z|=\delta_2\,|z|
    \]
   and \eqref{eq:boundzn0overz} holds.
   
   Set $\alpha = 2|\mathfrak M^{n_0}(\onesp)|$ (note that by definition of $n_0$, we have $\mathfrak M^{n_0}(\onesp)>0$). 
By superadditivity, we have
\[
\PP(|Z_{n_0}|\geq \alpha |z| \mid Z_0=z)\leq \PP(|Z_{n_0}| \geq \alpha|z| \mid Z_0=|z|\onesp).
\]
Now Theorem~\ref{thm:LGN1} entails that 
 \[
\PP(|Z_{n_0}| \geq \alpha m\mid Z_0=m\onesp)= \PP\left(\left.\frac{|Z_{n_0}|}{m} \geq 2|\mathfrak M^{n_0}(\onesp)| \,\right|\,Z_0=m\onesp\right)\xrightarrow[m\to+\infty]{} 0.
 \]
   Using this, we conclude that for all $\eta\in (0,1)$, there exists $k_2>0$ such that, for all $|z|\geq k_2$,
   \[
   \PP(|Z_{n_0}|\geq \alpha |z|\mid Z_0=z)\leq \eta/2.
   \]
    and  \eqref{eq:boundzn0overz2} holds, which concludes the proof of the first part of the lemma.

    \medskip
    
    In addition, setting $\tau'_{\rho}=\inf\{n\geq 0,|Z_n|\geq \rho\}$, we have for all starting point in $z\in \NN^p\setminus\{0\}$ such that $|z|\leq \rho$, $\mathbb P(\tau'_{\rho}\wedge T_0<+\infty\mid Z_0=z)=1$ (by Assumption~\ref{asu:transience} and because $Z_n\in \NN^p$ almost surely), we deduce that, for all $\eta\in(0,1)$, there exists $n_1\geq 1$ such that
    \begin{align*}
    \mathbb P(\tau'_{\rho}\wedge T_0\leq n_1\mid Z_0=z)\geq 1-\eta,\quad \forall |z|\leq \rho.
    \end{align*}
    Using the strong Markov property at time $\tau'_{\rho}\wedge T_0$, we deduce that
    \begin{align*}
    \mathbb P\left(\tau_{\delta_0,r}\wedge T_0\leq n_0+n_1\mid Z_0=z\right)\geq (1-\eta)^2,\quad\forall z\in\mathbb N^p\setminus\{0\}.
    \end{align*}
    Since this is true for all $\eta\in(0,1)$ this concludes the proof of Lemma~\ref{lem:minoration}.
\end{proof}

We are now in position to conclude the proof of Theorem~\ref{thm:integrabilitycriterion}. Consider $\delta_0\in (0,\nicefrac1p]$ and $n_0\in \NN^*$ given by Lemma~\ref{lem:minoration}. According to Lemma~\ref{lem:capAn}, for all $\varepsilon\in (0,1)$, for all $\eta\in (0,1)$, there exists $r>0$ such that, for all $z\in U_{\delta_0}$ with $|z|\geq r$,
\begin{align}
\label{eq:eqstart}
\mathbb P\left(\forall n\geq 0,\ Z_{n+1}\in|(1-\varepsilon)\mathfrak M(Z_n),(1+\varepsilon)\mathfrak M(Z_n)|\mid Z_0=z\right)\geq 1-\eta/2.
\end{align}
In addition, Lemma~\ref{lem:minoration} entails that there exists $\rho>0$ such that, for all $|z|\geq \rho$,
\[
\mathbb P\left(\tau_{\delta_0,r}\leq n_0\mid Z_0=z\right)\geq 1-\eta/2,
\]
and hence, using the strong Markov property, we deduce that for all $|z|\geq \rho$,
\begin{align*}
\mathbb P\left(\tau_{\delta_0,r}\leq n_0\text{ and }\forall n\geq \tau_{\delta_0,r},\  Z_{n+1}\in|(1-\varepsilon)\mathfrak M(Z_n),(1+\varepsilon)\mathfrak M(Z_n)|\mid Z_0=z\right)\geq 1-\eta.
\end{align*}
Finally, this implies that 
\begin{align*}
\mathbb P\left(Z_{n_0}\neq 0\text{ and }\forall n\geq n_0,\  Z_{n+1}\in|(1-\varepsilon)\mathfrak M(Z_n),(1+\varepsilon)\mathfrak M(Z_n)|\mid Z_0=z\right)\geq 1-\eta.
\end{align*}
This concludes the proof of 
the first part of Theorem~\ref{thm:integrabilitycriterion}.

\smallskip
We now prove Corollary~\ref{cor:integrabilitycriterion}, the result holds true trivially if $z=0$. Hence, using the Markov property at time $\tau_{\delta_0,r}\wedge T_0$, we deduce from~\eqref{eq:eqstart} and the fact that $\mathbb P\left( \tau_{\delta_0,r}\wedge T_0<+\infty\right)=1$, that for all $z\in \mathbb N^p$,
\begin{align*}
\mathbb P\left( \tau_{\delta_0,r}\wedge T_0<+\infty\text{ and } \forall n\geq \tau_{\delta_0,r}\wedge T_0,\ Z_{n+1}\in|(1-\varepsilon)\mathfrak M(Z_n),(1+\varepsilon)\mathfrak M(Z_n)|\mid Z_0=z\right)\geq 1-\eta/2.
\end{align*}
Hence, 
\begin{align*}
\mathbb P\left(\exists N\geq 0\text{ such that, } \forall n\geq N,\ Z_{n+1}\in|(1-\varepsilon)\mathfrak M(Z_n),(1+\varepsilon)\mathfrak M(Z_n)|\mid Z_0=z\right)\geq 1-\eta/2.
\end{align*}
Since the left hand term does not depend on $\eta$, this conclude the proof of Corollary~\ref{cor:integrabilitycriterion}.

\smallskip 
For the second part of Theorem~\ref{thm:integrabilitycriterion}, we simply observe that, $\mathbb P(\cdot\mid Z_0=z)$-almost surely, for all $\varepsilon>0$, there exists $N\geq 0$ (random) such that for all $m,k\geq 0$, for all $n\geq N$, $Z_{n+m}\in|(1-\varepsilon)^m \mathfrak M^m(Z_n),(1+\varepsilon)^m \mathfrak M^m(Z_n)|$, hence
\begin{align*}
\frac{Z_{n+m+k}}{|Z_{n+m}|}&\in \left|\frac{(1-\varepsilon)^{m+k} \mathfrak M^{m+k}(Z_n)}{(1+\varepsilon)^m |\mathfrak M^m(Z_n)|},\frac{(1+\varepsilon)^{m+k} \mathfrak M^{m+k}(Z_n)}{(1-\varepsilon)^m |\mathfrak M^m(Z_n)|}\right|\\
&\subset \left|\frac{(1-\varepsilon)^{m+k} \mathfrak M^{m+k}(Z_n/|Z_n|)}{(1+\varepsilon)^m |\mathfrak M^m(Z_n/|Z_n|)|},\frac{(1+\varepsilon)^{m+k} \mathfrak M^{m+k}(Z_n/|Z_n|)}{(1-\varepsilon)^m |\mathfrak M^m(Z_n/|Z_n|)|}\right|\\
&\subset \left|\frac{(1-\varepsilon)^{m+k} }{(1+\varepsilon)^m }\inf_{u\in S}\frac{ \mathfrak M^{m+k}(u)}{|\mathfrak M^m(u)|},\frac{(1+\varepsilon)^{m+k} }{(1-\varepsilon)^m }\sup_{u\in S}\frac{ \mathfrak M^{m+k}(u)}{|\mathfrak M^m(u)|}\right|,
\end{align*} 
where the infimum and supremum should be understood component-wise.
But, according to Theorem~\ref{thm:eigenvalue}~(\ref{item:thmKrausev}), $\frac{ \mathfrak M^{m+k}(u)}{|\mathfrak M^m(u)|}$ converges uniformly in $u\in S$ toward $(\lambda^*)^k z^*$ when $m\to+\infty$, and hence, for all $\varepsilon'>0$, choosing $m$ large enough, we have, for all $n\geq N$,
\begin{align*}
\frac{Z_{n+m+k}}{|Z_{n+m}|}&\in \left|\frac{(1-\varepsilon)^{m+k} }{(1+\varepsilon)^m }(1-\varepsilon')(\lambda^*)^k z^*,\frac{(1+\varepsilon)^{m+k} }{(1-\varepsilon)^m }(1+\varepsilon')(\lambda^*)^k z^*\right|.
\end{align*}
We deduce that (again the $\liminf$ and $\limsup$ should be understood component-wise)
\begin{align*}
\frac{(1-\varepsilon)^{m+k} }{(1+\varepsilon)^m }(1-\varepsilon')(\lambda^*)^k z^*\leq \liminf_{n\to+\infty} \frac{Z_{n+m+k}}{|Z_{n+m}|}\leq \limsup_{n\to+\infty} \frac{Z_{n+m+k}}{|Z_{n+m}|}\leq \frac{(1+\varepsilon)^{m+k} }{(1-\varepsilon)^m }(1+\varepsilon')(\lambda^*)^k z^*.
\end{align*}
Taking, first the limit when $\varepsilon\to 0$ ($m$ depends on $\varepsilon'$ but not on $\varepsilon$), then the limit when $\varepsilon'\to 0$, concludes the proof of Theorem~\ref{thm:integrabilitycriterion}.

\section{Identification of an intrinsic supermartingale and proof of Theorem~\ref{thm:type_profiles}}
\label{sec:supermartingale}

In this section we prove Theorem~\ref{thm:type_profiles}. In order to do so, we start by introducing the process
\begin{equation}
\label{Cn}
C_n=\dfrac{\mathcal P (Z_n)}{(\lambda^*)^n},
\end{equation}
where $\mathcal P$ is the function defined by \eqref{eq:function_P}.

We claim that a convergence (almost surely or in $L^1$) of the process $(C_n)_{n\in\NN}$ to a non-negative random variable $\mathcal C$ implies the convergence of the process $\left(\frac{Z_n}{(\lambda^*)^n}\right)_{n\in\NN}$. On the event of extinction, the result is trivial. On the event of survival, as $\mathcal P$ is positively homogeneous by Theorem~\ref{thm:eigenvalue}, we have 
\[
\dfrac{Z_n}{(\lambda^*)^n}=C_n\dfrac{Z_n}{|Z_n|} \dfrac{1}{\mathcal P\left(\frac{Z_n}{|Z_n|}\right)}.
\]
Applying Theorem~\ref{thm:integrabilitycriterion} and if $C_n\to \mathcal C$ we deduce that 
\[
\dfrac{Z_n}{(\lambda^*)^n}\longrightarrow \mathcal C z^*\dfrac{1}{\mathcal P(z^*)}=\mathcal Cz^*,\text{ as }n\to+\infty,
\]
and so the proof is complete. We also remark that the convergence of the sequence $\left(\frac{Z_n}{(\lambda^*)^n}\right)_{n\in\NN}$ is of the same type as the convergence of $(C_n)_{n\in\NN}$.

We divide this section into five parts. The first and second parts are dedicated to the proof of both results in Theorem~\ref{thm:type_profiles}. In the third part, we prove the convergence in $L^1$ under stronger assumptions and prove that, under these assumptions, the limit is non-degenerate at $0$. In the fourth part, we prove Proposition~\ref{prop:suffconditions_convergence} and Proposition~\ref{lem:cows_bulls}.

\subsection{Almost sure convergence.} In this section, we prove the first statement of Theorem~\ref{thm:type_profiles}. The proof comes from the following result.

\begin{lem}\label{lem:surmartingale} The sequence $(C_n)_{n\in\mathbb{N}}$ is a supermartingale with respect to $(\mathcal{F}_n)_{n\in \NN}$, the natural filtration of $(Z_n)_{n\in\NN}$.
\end{lem}

\begin{proof}
First, from the definition  of $\mathcal P$ and Theorem~\ref{thm:eigenvalue}, it follows that $\mathcal P \circ \mathfrak M = \lambda^* \mathcal P$. Then we recall that, by Lemma~\ref{lem:domM}, we have for all $n\in\NN$,
\[\EE(Z_{n+1} | \mathcal{F}_n) \leq \mathfrak M(Z_n).\]

Hence, Jensen's inequality together with the fact that $\mathcal P$ is increasing (since $\mathfrak M$ is) and concave imply that
\[\EE(C_{n+1}|\mathcal{F}_n)=\dfrac{\EE(\mathcal P(Z_{n+1})|\mathcal{F}_n)}{(\lambda^*)^{n+1}}\leq \dfrac{\mathcal P(\EE(Z_{n+1}|\mathcal{F}_n))}{(\lambda^*)^{n+1}}\leq \dfrac{\mathcal P(\mathfrak M(Z_n))}{(\lambda^*)^{n+1}}=\dfrac{\mathcal P(Z_n)}{(\lambda^*)^n}=C_n.\]

\end{proof}
    
    Since $(C_n)_{n\in\NN}$ is a non-negative supermartingale, then it exists a non-negative random variable $\mathcal C$ such that $C_n\xrightarrow[n\to+\infty]{\text{a.s.}}\mathcal C$.

    We now prove the convergence of the process $\left(\dfrac{W_n}{(\lambda^*)^n}\right)_{n\in\NN}$ in Theorem~\ref{thm:type_profiles}.
    We follow the proof for the single-type case presented in \cite{gonzalez1996limit}
    
    Let $N_n=\dfrac{Z_n}{(\lambda^*)^n}, S_n=\dfrac{W_n}{(\lambda^*)^n}$ and $ \tilde{\mathcal F}_n=\sigma (Z_0, V^{(k,\ell)}_{i,j}, \, 1\leq i\leq p, 1\leq j \leq q, k\in \NN^*, 1\leq \ell \leq n).$ We consider $Z_0=z$ and define for $n\geq 1$, $\hat{S}_{n}=(\hat{S}_{n,1},\dots, \hat{S}_{n,q})$ with for $j\in\{1,\dots,q\}$,
    
    \[\hat{S}_{n,j} = \dfrac{1}{(\lambda^*)^n}\sum_{i=1}^p\sum_{k=1}^{Z_{n-1,i}}V^{(k,n)}_{i,j}\mathds{1}_{V^{(k,n)}_{i,j}\leq (\lambda^*)^{n-1}}\]
    Note that for $j\in\{1,\dots,q\},$
    \begin{equation}
    \label{S_hat}
    \mathbb E(\hat{S}_{n+1,j}\mid  \tilde{\mathcal F}_n)=\dfrac{1}{\lambda^*} \sum_{i=1}^pN_{n,i}\mathbb E (V_{i,j}\mathds{1}_{V_{i,j}\leq (\lambda^*)^n}).
    \end{equation}
    Since $\lambda^*>1$, we have that $V_{i,j}\mathds{1}_{V_{i,j}\leq (\lambda^*)^n}\to V_{i,j}$ almost surely as $n\to +\infty$. Hence, by the Monotone Convergence Theorem, $\EE(V_{i,j}\mathds{1}_{V_{i,j}\leq (\lambda^*)^n})\to \EE(V_{i,j}).$ Then, taking $n\to +\infty$, we have that 
    \begin{equation}
    \label{convergence_expectation}
    \mathbb E(\hat{S}_{n+1}|\tilde{\mathcal F}_n) \longrightarrow \dfrac{1}{\lambda^*}\mathcal{C}z^*\mathbb V\text{ a.s.}
    \end{equation}
    
    We consider now the martingale given by
    \[\left(\sum_{n=1}^m \left(\hat{S}_{n}-\mathbb E(\hat{S}_{n}\mid \tilde{\mathcal F}_{n-1})\right)\right)_{m\in \mathbb N}.\]
    We have that for all $j\in\{1,\dots, q\}$ and $n\in \mathbb N$
    \begin{align*}
    \text{Var} \left(\hat{S}_{n+1,j} - \mathbb E (\hat{S}_{n+1,j}\mid \tilde{\mathcal F}_n)\right)& = \text{Var} \left(\dfrac{1}{(\lambda^*)^{n+1}}\sum_{i=1}^p\sum_{k=1}^{Z_n,i}\left(V_{i,j}^{(k,n)}\mathds{1}_{V_{i,j}^{(k,n)}\leq (\lambda^*)^n} - \mathbb E\left(V_{i,j}^{(k,n)}\mathds{1}_{V_{i,j}^{(k,n)}\leq (\lambda^*)^n}\right)\right)\right)\\
    &=\dfrac{1}{(\lambda^*)^{-2(n+1)}}\sum_{i=1}^p \mathbb E(Z_{n,i}) \text{Var}\left(V_{i,j}\mathds{1}_{V_{i,j}\leq (\lambda^*)^n}\right)\\
    &\leq (\lambda^*)^{-n-2} \sum_{i=1}^p \dfrac{\mathfrak M^n_i(z)}{(\lambda^*)^n} \text{Var}\left(V_{i,j}\mathds{1}_{V_{i,j}\leq (\lambda^*)^n}\right)\\
    &\leq (\lambda^*)^{-n-2} \sum_{i=1}^p \dfrac{\mathfrak M^n_i(z)}{(\lambda^*)^n}\int_0^{+\infty}x^2 \mathds{1}_{x\leq (\lambda^*)^n} \,\mathrm dF_{i,j}(x),
    \end{align*}
    where $F_{i,j}(x)=\mathbb P (V_{i,j}\leq x).$
    
    Since $\left(\dfrac{\mathfrak M^n(z)}{(\lambda^*)^n}\right)_{n\in \NN}$ is convergent, then it is bounded, and so there exists $C>0$ such that
    
    \begin{align*}
    \sum_{n\in \NN} \text{Var}\left(\hat{S}_{n+1,j} - \mathbb E(\hat{S}_{n+1,j}\mid \tilde{\mathcal F}_n)\right) & \leq C \sum_{i=1}^p \int_0^{+\infty} x^2 \sum_{n\in\NN} \dfrac{1}{(\lambda^*)^n}\mathds{1}_{x\leq (\lambda^*)^n} \,\mathrm dF_{i,j}(x)\\
    &=C\sum_{i=1}^p\int_0^{+\infty} x^2O(x^{-1}) \,\mathrm dF_{i,j}(x) <+\infty,
    \end{align*}
    and so applying the Martingale Convergence Theorem, we have that $\sum_{n\in\NN} \left(\hat{S}_{n+1,j}-\mathbb E(\hat{S}_{n+1,j}\mid \tilde{\mathcal F}_n) \right)$
    is convergent almost surely and in $L^1$. This implies that $\hat{S}_{n+1,j} - \mathbb E(\hat{S}_{n+1,j}\mid \tilde{\mathcal F}_n)\to 0$ almost surely. Thanks to \eqref{convergence_expectation}, we have that $\hat{S}_{n+1}\to (\lambda^*)^{-1}\mathcal C z^* \mathbb V$. We finish by proving that $(S_n)_{n\in\NN}$ is an equivalent sequence. Fix $j\in \{1,\dots,q\}$
    \begin{align*}
    \sum_{n\in\NN} \mathbb P (S_{n+1,j}\neq \hat{S}_{n+1,j}) & =\sum_{n\in\NN} \mathbb E(\mathbb P (\exists i \leq p, \exists k\leq Z_{n,i}, V_{i,j}^{(k,n)}>(\lambda^*)^n\mid \tilde{\mathcal F}_n))\\
    &\leq \sum_{n\in\NN}\sum_{i=1}^p \mathbb E\left( \sum_{k=1}^{Z_{n,i}} \mathbb P (V_{i,j}>(\lambda^*)^n)\right)\\
    &\leq C\sum_{i=1}^p \sum_{n\in\NN} (\lambda^*)^n\mathbb P(V_{i,j}>(\lambda^*)^n)\\
    &\leq C\sum_{i=1}^p \int\limits_0^{+\infty} \sum_{n\in\NN} (\lambda^*)^n \mathds{1}_{x>(\lambda^*)^n}\,\mathrm dF_{i,j}(x)\\
    &=\sum_{i=1}^p \int\limits_0^{+\infty} O(x)\,\mathrm dF_{i,j}(x) <+\infty.
    \end{align*}
    The conclusion then follows by the Borel-Cantelli lemma.

\subsection{Extinction vs $\{\mathcal C=0\}$.}\label{sec:extinction_event} In this section, we prove the second part of Theorem~\ref{thm:type_profiles}.

The inclusion $\{\exists n\in\mathbb N,\:Z_n=0\} \subset \{\mathcal C=0\}$ is obvious. We then consider the case $\lambda^*>1$ and show that $\{\forall n,\ Z_n\neq 0)\subset \{\mathcal C>0\}$.

For all $\varepsilon>0$, we set 
\[
\tau_{\varepsilon}=\inf\{n\geq 0,\ \frac{\mathcal P(Z_n)}{(\lambda^*)^n}\leq \varepsilon\}.
\]

By assumption $\mathcal C$ is non-degenerate at $0$ for all $z$ such that $q_z<1$. Hence, for all $z$ such that $q_z<1$, there exists $\varepsilon_z>0$ such that 
\[
\mathbb P(\tau_{\varepsilon_z}=+\infty\mid Z_0=z)>0.
\]

In addition, Theorem~\ref{thm:integrabilitycriterion} entails that there exists $r_0\in\mathbb N$ such that, for all $|z|\geq r_0$, $q_z<1$. Hence,
setting
\[
\varepsilon_0=\min_{i=1,\ldots,p} \varepsilon_{r_0 e_i}
\]
and using the fact that, by superadditivity, $\mathbb P(\tau_{\varepsilon_0}=+\infty\mid Z_0=z)$ is increasing with $z$, we deduce that, for all $|z|\geq pr_0$,
\begin{align*}
\mathbb P(\tau_{\varepsilon_0}=+\infty\mid Z_0=z)&\geq a_0:=\min_{i=1,\ldots,p} \mathbb P(\tau_{\varepsilon_0}=+\infty\mid Z_0=r_0 e_i)\\
&\geq \min_{i=1,\ldots,p} \mathbb P(\tau_{\varepsilon_{r_0e_i}}=+\infty\mid Z_0=r_0 e_i)>0.
\end{align*}

We define $\tau^0=0$, $\tau^1=\tau_{\varepsilon_0}$ and, for all $n\geq 1$,
\[
\tau^{n+1}=\inf\{n\geq \tau^n+1,\ \frac{\mathcal P(Z_n)}{(\lambda^*)^n}\leq \varepsilon_0\}.
\]
Then we obtain, for all $|z|\geq pr_0$ and all $n\geq 1$,
\begin{align}
&\mathbb P(\{\tau^{n+1}<+\infty\}\cap \{|Z_m|\geq p r_0,\ \forall m\in[0,\tau^n]\}\mid Z_0=z)\nonumber\\
&\qquad\qquad\leq \mathbb E\left(1_{\tau^n<+\infty,\ |Z_m|\geq p r_0,\ \forall m\in[0,\tau^{n}]}\,\mathbb P(\tau_{\varepsilon_0}<+\infty\mid Z_0=z')_{\vert {z'=Z_{\tau^n}}}\mid Z_0=z\right)\nonumber\\
&\qquad\qquad\leq \mathbb P(\{\tau^{n}<+\infty\}\cap \{|Z_m|\geq p r_0,\ \forall m\in[0,\tau^{n-1}]\}\mid Z_0=z)\,(1-a_0)\nonumber\\
&\qquad\qquad\leq \ldots \leq (1-a_0)^{n+1}.\label{eq:Cdege1}
\end{align}
In addition, according to Theorem~\ref{thm:LGN1} and Theorem~\ref{thm:integrabilitycriterion}, for all $\eta\in (0,1)$, there exists $r_\eta>0$ such that, for all $|z|\geq r_\eta$,
\begin{align*}
\mathbb P(|Z_n|\geq p r_0,\ \forall n\geq 0\mid Z_0=z)\geq 1-\eta.
\end{align*}
Using this last inequality and~\eqref{eq:Cdege1}, we deduce that, for all $|z|\geq r_\eta$ and all $n\geq 1$,
\begin{align*}
\mathbb P(\tau^{n+1}<+\infty\mid Z_0=z)&\leq \mathbb P(\{\tau^{n+1}<+\infty\}\cap \{|Z_n|\geq pr_0,\ \forall n\in[0,\tau^n]\}\mid Z_0=z)\\
&\qquad\qquad + \mathbb P(\exists n\geq 0 \text{ such that } |Z_n|< pr_0\mid Z_0=z)\\
&\leq (1-a_0)^{n+1}+\eta.
\end{align*}
Since $\{\mathcal C=0\}\subset \{\tau^{n+1}<+\infty,\ \forall n\geq 1\}$, we deduce that, for all $|z|\geq r_\eta$,
\begin{align}
\label{eq:Cdege2}
\mathbb P(\mathcal C=0\mid Z_0=z)&\leq \eta.
\end{align}
Denoting $T_\eta:=\inf\{n\geq 0,\ |Z_n|\geq r_\eta\}$, we deduce from the transitivity assumption that $\{T_\eta<+\infty\}\supset \{\forall n,\ Z_n\neq 0\}$.
Hence, for all $z\in \mathbb N^p$, using the strong Markov property at time $T_\eta$ and then~\eqref{eq:Cdege2},
\begin{align*}
\mathbb P(\mathcal C=0\text{ and }\forall n,\ Z_n\neq 0\mid Z_0=z)&\leq \mathbb P(\mathcal C= 0\text{ and }T_\eta<+\infty\mid Z_0=z)\\
&=\mathbb E\left(1_{T_\eta<+\infty}\,\mathbb P\left(\lim_{n\to+\infty}\frac{\mathcal P(Z_n)}{(\lambda^*)^{n+u}}= 0\mid Z_0=z'\right)_{\vert u=T_\eta,\,z'={Z_{T_\eta}}}\mid Z_0=z\right)\\
&=\mathbb E\left(1_{T_\eta<+\infty}\mathbb P\left(\mathcal C=0\mid Z_0=z'\right)_{\vert {z'=Z_{T_\eta}}}\mid Z_0=z\right)\leq \eta.
\end{align*}

Since the last inequality holds true for all $\eta\in(0,1)$, we deduce that, for all $z\in\mathbb N^p$,
\begin{align*}
\mathbb P(\mathcal C=0\text{ and }\forall n,\ Z_n\neq 0\mid Z_0=z)=0.
\end{align*}

\subsection{$L^1$ convergence and non-degeneracy of the limit.} \label{sec:proofpropL1-condition} In this section, we prove the convergence of $(C_n)_{n\in\NN}$ in $L^1$ to a non-degenerate limit $\mathcal C$, which corresponds to Proposition~\ref{porp:L1-condition}. We recall that in this part we assume that there exists a concave monotone increasing function $U:\mathbb R_+ \longrightarrow \mathbb R_+$, such that for all $y\in\RR_+$,
\[\sup_{z\in \RR_+^p: \mathcal P(z)=y} \mathbb E(|\mathcal P (Z_1) - \mathcal P(\mathfrak M(\lfloor z\rfloor ))|\mid Z_0=\lfloor z\rfloor )\leq U(y),\]
with $y\to\nicefrac{U(y)}{y}$ non-increasing and
\[\int\limits_1^{+\infty} \dfrac{U(y)}{y^2}\,\mathrm dy<+\infty.\]

The idea behind the proof is to use the following lemma.

\begin{lem}\label{lem:klebaner_1984}[See \cite{klebaner1984geometric} - Lemma 2] Let $f:\RR_+\longrightarrow \RR_+$ be a non-increasing function, such that $x\mapsto xf(x)$ is non-decreasing and
    \[\sum_{n=1}^{+\infty} \dfrac{f(n)}{n}<+\infty.
    \]
    Let $(a_n)_{n\in\NN}$ be a sequence of positive numbers satisfying for some $m>1$ and all $n\geq 0$
    \[|a_{n+1}-a_n|\leq a_nf(a_nm^n).
    \]
    Then 
\begin{itemize}
    \item $\lim\limits_{n\to+\infty}a_n=a$ exists,
    \item there exists a constant $b_0$ depending only on the function $f$ and $m$ such that if $a_0>b_0$ then $a>0$.
\end{itemize}
\end{lem}

We start our proof by noticing that, for all $z\in \NN^p$,
\begin{align*}
|\mathbb E(C_{n+1}\mid Z_0=z) - \mathbb E(C_n\mid Z_0=z)| &\leq \mathbb E(|C_{n+1}-C_n|\mid Z_0=z)\\
& = \mathbb E(\mathbb E(|C_{n+1} - C_n| \mid \mathcal F_n)\mid Z_0=z)\\
&=\dfrac{1}{(\lambda^*)^{n+1}} \mathbb E(\mathbb E(|\mathcal P(Z_{n+1})-\lambda^* \mathcal P(Z_n)| \mid \mathcal F_n)\mid Z_0=z)\\
&\leq \dfrac{1}{(\lambda^*)^{n+1}}\mathbb E(U(\mathcal P(Z_n))\mid Z_0=z),
\end{align*}
and so applying Jensen's inequality, since $U$ is concave, we get

\begin{equation}
\label{inequality_2}
\mathbb E (|C_{n+1} - C_n|\mid Z_0=z) \leq (\lambda^*)^{-1} \mathbb E(C_n\mid Z_0=z) \dfrac{U(\mathbb E(\mathcal P(Z_n)\mid Z_0=z))}{\mathbb E(\mathcal P(Z_n)\mid Z_0=z)}
\end{equation}

From this inequality, we have that if we define
\[F(x) = \dfrac{U(x)}{\lambda^* x},\]
then, we obtain,
\[
|\mathbb E(C_{n+1}\mid Z_0=z) - \mathbb E(C_n\mid Z_0=z)| \leq \mathbb E(C_n\mid Z_0=z) F(\mathbb E(C_n\mid Z_0=z)(\lambda^*)^n),\]
and we can apply Lemma~\ref{lem:klebaner_1984} with $f=F$ and $m=\lambda^*$.  This implies that, for all $z\in \NN^p$,
\[
c(z):=\lim_{n\to+\infty} \mathbb E(C_n|Z_0=z)
\]
exists. It also implies that there exists $b_0$ such that if $\mathcal P(z)\geq b_0$, then $c(z)>0$. Since $\mathcal P$ is homogeneous and lower bounded away from $0$ on $S$, we deduce that there exists $r_0>0$ such that, for all $z\in\NN^p$ with $|z|\geq r_0$, $c(z)>0$. Since in addition $c(z)$ is increasing with $z$, we deduce that 
\[
\underline c:= \inf_{z\in\NN^p,|z|\geq r_0} c(z)>0.
\]

If we now define 
\[
T=\inf \{n\in\NN\::\:|Z_n|\geq r_0\},
\]
we have that, if $z\in\NN^p$ is such that $q_z<1$, then $\PP(T<+\infty\mid Z_0=z)\geq q_z>0$ by transitivity assumption and so, applying the strong Markov property we obtain
\begin{align*}
c(z)=\lim_{n\to\infty}\EE(C_n|Z_0=z)&\geq \lim_{n\to+\infty}\EE\left(\EE(C_n\mid Z_0=y)_{\rvert y= Z_T}(\lambda^*)^{-T}\mathds{1}_{T<+\infty}\mid Z_0=z\right)\\
&\geq \underline{c}\,\EE\left((\lambda^*)^{-T}\mathds{1}_{T<+\infty}\mid Z_0=z\right)>0.
\end{align*}

Now fix $z\in \NN^p$ such that $q_z<1$ and  take $\varepsilon$ such that $c(z)-\varepsilon>0$. We can find $N_0$ such that for all $n\geq N_0$, 
\[ c(z)-\varepsilon \leq \mathbb E(C_n\mid Z_0=z) \leq c(z)+\varepsilon.\]

Hence, using this in \eqref{inequality_2}, we get that for all $n\geq N_0$, since $x\to \nicefrac{\hat U(x)}{x}$ is non-increasing,

\[\mathbb E(|C_{n+1}-C_n| \mid Z_0=z) \leq (\lambda^*)^{-1} ( c(z)+\varepsilon) \dfrac{U(( c(z)+\varepsilon)(\lambda^*)^n)}{( c(z)-\varepsilon)(\lambda^*)^n},\]
and so we can find $C>0$ and $\delta >0$, such that for all $n\in \NN$,

\[\mathbb E(|C_{n+1}-C_n| \mid Z_0=z) \leq C\dfrac{U(\delta (\lambda^*)^n)}{(\lambda^*)^n}.\]

On the other hand, since the integral $\displaystyle\int_1^{+\infty} \frac{U(y)}{y^2}\,\mathrm dy$ is finite, we have that the series 
\[\sum_{n\in \NN} \dfrac{U(\delta (\lambda^*)^n)}{(\lambda^*)^n}\]
converges (see \cite{klebaner1985limit}). This implies that $(C_n)_{n\in \NN}$ is a Cauchy sequence in $L^1$, which gives the $L^1$ convergence. 

Finally, we have that if $z\in \NN^p$ is such that $q_z<1$, then
\[\mathbb E(\mathcal C|Z_0=z) = \lim_{n\to +\infty} \mathbb E(C_n|Z_0=z)>0,\]
which proves that the limit is non-degenerate at $0$.

\subsection{Sufficient Conditions for the existence of $U$.} \label{sec:supermartingale_suff} In this section, we prove Proposition~\ref{prop:suffconditions_convergence} and Proposition~\ref{lem:cows_bulls}. We start the proof by stating and proving a lemma that is useful for the proof of both results. 

\begin{lem}
\label{lem:Expect_LlogL} Consider $p$ non-negative independent and integrable random variables  $X_1,\dots, X_p$. Set $z\in \NN^p$ and $\beta >0$, then 
\[\mathbb E\left( \left|\sum_{i=1}^p\sum_{k=1}^{z_i} \left(X_i^{(k)} - \mathbb E(X_i)\right)\right|\right) \leq |z|^{\beta+\nicefrac{1}{2}} + 2|z|\sum_{i=1}^p\int\limits_{|z|^\beta}^{+\infty}x\,\mathrm dF_i(x),\]
where $(X_1^{(k)},\ldots,X_p^{(k)})_{k\in \NN}$ are i.i.d. copies of $(X_1,\ldots,X_p)$, and $F_i(x) = \mathbb P(X_i\leq x)$.
\end{lem}

\begin{proof} We have 
\begin{align*}
\mathbb E\left( \left|\sum_{i=1}^p\sum_{k=1}^{z_i} \left(X_i^{(k)} - \mathbb E(X_i)\right)\right|\right)&\leq  \mathbb E\left( \left|\sum_{i=1}^p\sum_{k=1}^{z_i} \left(X_i^{(k)}\mathds{1}_{X_i^{(k)} \leq |z|^{\beta}} - \mathbb E(X_i\mathds{1}_{X_i \leq |z|^{\beta}})\right)\right|\right)\\
&+\mathbb E\left( \left|\sum_{i=1}^p\sum_{k=1}^{z_i} \left(X_i^{(k)}\mathds{1}_{X_i^{(k)}> |z|^{\beta}} - \mathbb E(X_i\mathds{1}_{X_i > |z|^{\beta}})\right)\right|\right).
\end{align*}

We bound the two expectations above separately. For the first one we have, setting $Y_i^{(k)} = X_i^{(k)}\mathds{1}_{X_i^{(k)}\leq |z|^{\beta}}$,
\begin{align*}
\mathbb E \left(\left|\sum_{i=1}^p \sum_{k=1}^{z_i} \left(Y_i^{(k)}- \mathbb E(Y_i^{(k)})\right)\right|\right)^2 & \leq \mathbb E \left(\left(\sum_{i=1}^p \sum_{k=1}^{z_i} \left(Y_i^{(k)}- \mathbb E(Y_i^{(k)})\right)\right)^2\right)\\
&=\sum_{i=1}^p\sum_{k=1}^{z_i} \mathbb E\left(\left(Y_i^{(k)}-\mathbb E(Y_i^{(k)})\right)^2\right)\\
&=\sum\limits_{i=1}^p \sum\limits_{k=1}^{z_i} \text{Var}\left(X_i^{(k)}\mathds{1}_{X_i^{(k)}\leq |z|^{\beta}}\right)\leq\sum_{i=1}^p\sum_{k=1}^{z_i}|z|^{2\beta}\\
&= |z|^{2\beta +1}.
\end{align*}
For the second term, we have 
\begin{align*}
\mathbb E \left(\left|\sum_{i=1}^p \sum_{k=1}^{z_i} \left(X_i^{(k)}\mathds{1}_{X_i^{(k)}>|z|^{\beta}}- \mathbb E(X_i^{(k)}\mathds{1}_{X_i^{(k)}>|z|^{\beta}})\right)\right|\right)&\leq 2\sum_{i=1}^p\sum_{k=1}^{z_i} \mathbb E\left(X_i^{(k)}\mathds{1}_{X_i^{(k)}>|z|^{\beta}}\right)\\
&=2\sum_{i=1}^p z_i \int\limits_0^{+\infty} x \mathds{1}_{x>|z|^{\beta}} \,\mathrm dF_i(x)\\
&\leq 2|z|\sum_{i=1}^p \int\limits_{|z|^{\beta}}^{+\infty} x\,\mathrm dF_i(x),
\end{align*}
and the result follows.
\end{proof}

The following lemma is a key ingredient of the proof.

\begin{lem}\label{lem:Klebaner_1985}[See \cite{klebaner1985limit} - Page 52] Let $f$ be a positive function on $[1,+\infty)$ such that $x\mapsto\frac{f(x)}{x}$ is non-increasing and
\[\int\limits_1^{+\infty} \dfrac{f(x)}{x^2}\,\mathrm dx<+\infty.\]
Then there exists a monotone increasing function $\hat{f}$ such that for all $x\in \RR_+,\ \hat{f}(x)\geq f(x)$, $x\to\nicefrac{\hat{f}(x)}{x}$ is non-increasing, $\int_1^{+\infty}(\nicefrac{\hat{f}(x)}{x^2})\,\mathrm dx<+\infty$ and $\hat{f}$ is concave on $\RR_+$.
\end{lem}

\begin{proof}[Proof of Proposition~\ref{prop:suffconditions_convergence}] We recall that for this proof, we have the following additional assumption.

\begin{asu}
\label{asu:L1_convergence}
\begin{enumerate} We assume that
        \item The mating function $\xi$ and the function $\mathcal P$ are Lipschitz.
        \item For all $i\in\{1,\dots p\}$ and $j\in\{1,\dots,q\}$, we have  $\mathbb E(V_{i,j}\log V_{i,j})<+\infty$.
        \item There exists $\alpha>0$ such that, $\forall z\in\RR_+^p\setminus\{0\}$,
        \begin{equation}
        \label{speed_condition}
        \left|\dfrac{\xi(z\mathbb V)}{|z|}- \dfrac{\mathfrak M(z)}{|z|}\right|=O(|z|^{-\alpha}).
        \end{equation} 
\end{enumerate}
\end{asu}

Since $\mathcal P$ is continuous, positive on $\RR_+^p\setminus \{0\}$ and positively homogeneous, there exist $L_1,L_2>0$ such that, for all $z\in \RR_+^p$, 
\begin{equation}
\label{bound_P}
L_1|z|\leq \mathcal P(z)\leq L_2|z|.
\end{equation}

Let us consider first $z\in \NN^p$. Since $\mathcal P$ is Lipschitz, there exists $K_1>0$ such that
\begin{align*}
\mathbb E(|\mathcal P(Z_1) - \lambda^*\mathcal P (z)|\mid Z_0=z)&=\mathbb E(|\mathcal P(Z_1) - \mathcal P(\mathfrak M(z))| \mid Z_0=z)\\
&\leq K_1\mathbb E\left(\left.\left|Z_1 - \mathfrak M(z)\right|\ \right| Z_0=z\right)\\
&\leq K_1\left(\mathbb E\left(\left|\left. Z_1-\xi(z\mathbb V)\right| \ \right|Z_0=z \right) + |z|\left|\dfrac{\xi(z\mathbb V)}{|z|} - \dfrac{\mathfrak M(z)}{|z|}\right|\right).
\end{align*}
Using that $\xi$ is also Lipschitz, there exists $K_2>0$ such that
\begin{align*}
\mathbb E\left(\left|\left. Z_1-\xi(z\mathbb V)\right| \ \right|Z_0=z \right) & \leq K_2\sum_{j=1}^q\mathbb E\left(\left|\sum_{i=1}^p\sum_{k=1}^{z_i} V_{i,j}^{(k)} - (z\mathbb V)_j\right|\right)\\
&=K_2\sum_{j=1}^q \mathbb E \left(\left|\sum_{i=1}^p \sum_{k=1}^{z_i} \left(V_{i,j}^{(k)}- \mathbb E(V_{i,j}^{(k)})\right)\right|\right)
\end{align*}
Making use of Lemma~\ref{lem:Expect_LlogL} with $\beta \in (0,\nicefrac{1}{2})$ and applying Assumption~\ref{asu:L1_convergence}-3, we obtain that there exists $K_3$ such that for all $z\in \NN^p$,
\[\mathbb E(|\mathcal P(Z_1) - \lambda^*\mathcal P (z)|\mid Z_0=z) \leq K_3\left(|z|^{\beta+\nicefrac{1}{2}}+|z|\sum_{i=1}^p\sum_{j=1}^q\int\limits_{|z|^{\beta}}^{+\infty} x\,\mathrm dF_{i,j}(x)+|z|^{1-\alpha}\right),\]
for $\alpha>0$ and $F_{i,j}(x) = \mathbb P(V_{i,j}\leq x)$.

To extend the previous bound to $z\notin \NN^p$, note that if $z\in\RR_+^p$ with $|z|>2p$, we have $0< \frac{1}{2}|z|\leq |\lfloor z \rfloor |\leq |z|$. Hence there exist $K_4,K_5>0$ such that for all $z\in\RR_+^p$ with $|z|>2p$,
\[\mathbb E(|\mathcal P(Z_1) - \lambda^*\mathcal P (\lfloor z\rfloor )|\mid Z_0=\lfloor z\rfloor )\leq K_4\left(|z| \sum_{i=1}^p \sum_{j=1}^q \int\limits_{K_5|z|^{\beta}}^{+\infty} x\,\mathrm dF_{i,j}(x) +|z|^{\beta+\nicefrac{1}{2}}+|z|^{1-\alpha}\right).\] 

Finally, applying \eqref{bound_P}, we get that there exists $C_1,C_2>0$, such that for all $z\in\RR^p_+$ with $|z|>2p$
\[\mathbb E(|\mathcal P(Z_1) - \lambda^*\mathcal P (\lfloor z\rfloor )|\mid Z_0=\lfloor z\rfloor )\leq C_1\left(\mathcal P(z) \sum_{i=1}^p \sum_{j=1}^q \int\limits_{C_2^{\beta}\mathcal P(z)^{\beta}}^{+\infty} x\,\mathrm dF_{i,j}(x)+ \mathcal P(z)^{\beta+\nicefrac{1}{2}}+\mathcal P(z)^{1-\alpha}\right).\]

Now define 
\[C_3:=\max_{z\in \NN^p: |z|\leq 2p} \mathbb E(|\mathcal P(Z_1)-\mathcal P(\lfloor z \rfloor ) | \mid Z_0=z)<+\infty.\]
Then, for all $z\in \RR_+^p$, 
\[\mathbb E(|\mathcal P(Z_1) - \lambda^*\mathcal P (\lfloor z\rfloor )|\mid Z_0=\lfloor z\rfloor )\leq C_1\left(\mathcal P(z) \sum_{i=1}^p \sum_{j=1}^q \int\limits_{C_2^{\beta}\mathcal P(z)^{\beta}}^{+\infty} x\,\mathrm dF_{i,j}(x)+ \mathcal P(z)^{\beta+\nicefrac{1}{2}}+\mathcal P(z)^{1-\alpha}\right)+C_3.\]
This implies that for all $y\in \RR_+$,
\[\sup_{z: \mathcal P (z) = y} \mathbb E(|\mathcal P(Z_1) - \lambda^*\mathcal P (\lfloor z\rfloor )|\mid Z_0=\lfloor z\rfloor ) \leq C_1\left(y\sum_{i=1}^p \sum_{j=1}^q \int\limits_{C_2^{\beta} y^{\beta}}^{+\infty}x\,\mathrm dF_{i,j}(x) +y^{\beta+\nicefrac{1}{2}}+y^{1-\alpha}\right)+C_3.\]

Now set $F:\RR_+\longrightarrow \RR_+$ given by
\[F(y)=C_1\left(y\sum_{i=1}^p \sum_{j=1}^q \int\limits_{C_2^{\beta} y^{\beta}}^{+\infty}x\,\mathrm dF_{i,j}(x) +y^{\beta+\nicefrac{1}{2}}+y^{1-\alpha}\right)+C_3.\]
Then we have on one hand that
\[\dfrac{F(y)}{y}=C_1\left(\sum_{i=1}^p \sum_{j=1}^q \int\limits_{C_2^{\beta} y^{\beta}}^{+\infty}x\,\mathrm dF_{i,j}(x) +y^{\beta-\nicefrac{1}{2}}+y^{-\alpha}\right)+\dfrac{C_3}{y},\]
defines a non-increasing function on $(0,+\infty)$ since $\beta \in (0,\nicefrac{1}{2})$ and $\alpha>0$.
On the other hand
\[\int\limits_1^{+\infty}\dfrac{F(y)}{y^2}\,\mathrm dy = C_1\left(\sum_{i=1}^p\sum_{j=1}^q \int\limits_1^{+\infty}\int\limits_{C_2^{\beta}y^{\beta}}^{+\infty}\dfrac{x}{y}\,\mathrm dF_{i,j}(x)\,\mathrm dy + \int\limits_1^{+\infty} \left(\dfrac{1}{y^{\nicefrac{3}{2}-\beta}}+\dfrac{1}{y^{1+\alpha}}\right)\,\mathrm dy\right)+\int\limits_1^{+\infty} \dfrac{C_3}{y^2}\,\mathrm dy.\]

Once again, since $\alpha>0$ and $\beta\in (0,\nicefrac{1}{2})$, we only need to prove that the first integral is finite. In fact, applying Fubini's Theorem,
\begin{align*}
\int\limits_1^{+\infty}\int\limits_{C_2^{\beta}y^{\beta}}^{+\infty}\dfrac{x}{y}\,\mathrm dF_{i,j}(x)\,\mathrm dy&\leq  \int\limits_0^{+\infty}x\int\limits_1^{\frac{1}{C_2}x^{\nicefrac{1}{\beta}}} \dfrac{\mathrm dy}{y} \,\mathrm dF_{i,j}(x)\\
&=\int\limits_0^{+\infty} x\,O(\log x)\,\mathrm dF_{i,j}(x),
\end{align*}
where the last integral is finite, since $\mathbb E(V_{i,j}\log V_{i,j})<+\infty$ by assumption. 

Applying Lemma~\ref{lem:Klebaner_1985}, there exists a concave monotone increasing function $U:\RR_+\longrightarrow \RR_+$, with $y\to \nicefrac{U(y)}{y}$ non-increasing and $\displaystyle\int_1^{+\infty} \frac{U(y)}{y^2}\,\mathrm dy<+\infty,$ and such that for all $y\in \RR_+$,
\[\sup_{z: \mathcal P(z) = y } \mathbb E(|\mathcal P (Z_1) - \lambda^*\mathcal P (\lfloor z \rfloor )| \mid Z_0=\lfloor z \rfloor ) \leq F(y) \leq U(y).\]
The proof is then complete.
\end{proof}

\begin{proof}[Proof of Proposition~\ref{lem:cows_bulls}] For this proof, we consider the case of the Promiscuous mating of Example~\ref{exa:cows_bulls}. We set $p=n_f$ and 
\[
\xi((x_1,\dots x_p),(y_1,\dots , y_{n_m})) = (x_1,\dots x_p)\prod_{j=1}^{n_m} \mathds{1}_{y_j>0}.
\]
Hence, we have $\mathfrak M (z) = z\mathbb X \mathds{1}_{z\mathbb Y>0}$. Since  $\mathbb Y_{i,j}>0$ for all $i\in\{1,\dots p\}$ and all $j\in\{1,\dots n_m\}$, we have $\mathfrak M(z) = z\mathbb X$. We also note that 
\[
L(z) = \lim_{n\to +\infty} \dfrac{\mathfrak M^n(z)}{(\lambda^*)^n} = \langle u^*,z\rangle z^*,
\]
where $u^*$ is the positive right eigenvector of $\XX$ such that $\langle u^*,z^*\rangle=1$ and $\langle \cdot , \cdot \rangle$ stands for the Euclidean product. In particular, $\mathcal P(z)=\langle u^*,z\rangle$.

Consider $z\in \NN^p$, then
\begin{align*}
\mathbb E(|\mathcal P (Z_1) - \lambda^*\mathcal P (z)| \mid Z_0 = z)&= \mathbb E(|\langle u^*, Z_1 \rangle - \langle u^* , z \mathbb X\rangle | \mid Z_0=z)\\
& = \mathbb E(|\langle u^*, Z_1 - z\mathbb X \rangle | \mid Z_0=z)\\
& \leq C_0\sum_{j=1}^p\mathbb E (|Z_{1.j} - (z\mathbb X)_j| \mid Z_0=z),
\end{align*}
for some constant $C_0>0$ and hence
\begin{align*}
\mathbb E(|\mathcal P (Z_1) - \lambda^*\mathcal P (z)| \mid Z_0 = z)
&\leq C_0 \sum_{j=1}^p\mathbb E\left(\left|\sum_{i=1}^p \sum_{k=1}^{z_i} X_{i,j}^{(k)}\mathds{1}_{\forall \ell \leq n_m,\:M_{1,\ell}>0} - \sum_{i=1}^p z_i \mathbb X_{i,j} \right|\right)\\
&\leq C_0\sum_{j=1}^p \mathbb E\left(\left|\sum_{i=1}^p\sum_{k=1}^{z_i} \left(X_{i,j}^{(k)} - \mathbb X_{i,j} \right)\right|\right)\\
&+C_0\sum_{j=1}^p\sum_{i=1}^p\sum_{k=1}^{z_i}\mathbb E\left(\left.X_{i,j}^{(k)}\mathds{1}_{\exists \ell \leq n_m,\: M_{1,\ell}=0}\right|Z_0=z\right).
\end{align*}
with $(X^{(k)})_{k\in\NN}$ a family of i.i.d copies of $X=(X_{i,j})_{1\leq i,j \leq p}$.
For the second term, we have for $i,j\in \{1,\dots, p\}$ and $k\in \{1,\dots z_i\}$ fixed
\begin{align*}
\mathbb E\left(\left.X_{i,j}^{(k)}\mathds{1}_{\exists \ell \leq n_m,\: M_{1,\ell}=0}\right|Z_0=z\right)& \leq \sum_{\ell=1}^{n_m}\mathbb E\left(\left.X_{i,j}^{(k)} \mathds{1}_{M_{1,\ell}=0}\right|Z_0=z\right)\\
&=\sum_{\ell=1}^{n_m} \mathbb E\left(X_{i,j}^{(k)} \prod_{i'=1}^p \prod_{k'=1}^{z_{i'}} \mathds{1}_{Y_{i',\ell}^{(k')}=0}\right)\\
&\leq \sum_{\ell=1}^{n_m} \mathbb X_{i,j} \prod_{i'=1}^{p}\prod_{\substack{k'=1\\(i',k')\neq (i,k)}}^{z_{i'}} \mathbb P \left(Y_{i',\ell}^{(k')}=0\right)\\
&\leq \mathbb X_{i,j} n_m \gamma^{|z|-1},
\end{align*}
with $\gamma=\max_{i'\leq p, \ell \leq n_m} \mathbb P(Y_{i',\ell}=0)\in [0,1)$.
Then, applying Lemma~\ref{lem:Expect_LlogL} with $\beta\in (0,\nicefrac{1}{2})$, we deduce that there exists $C_1>0$ such that, for all $z\in \NN^p$,
\[\mathbb E(|\mathcal P(Z_1) - \lambda^* \mathcal P(z)|\mid Z_0=z) \leq C_1\left(|z| \sum_{i=1}^p \sum_{j=1}^q \int\limits_{|z|^{\beta}}^{+\infty} x\,\mathrm dF_{i,j}(x)+ |z|^{\beta+\nicefrac{1}{2}}+|z|\gamma^{|z|-1}\right).\]

The result then follows as in the proof of Proposition~\ref{prop:suffconditions_convergence}.
\end{proof}

\appendix
\section{Transitivity condition under strong primitivity}
\label{sec:annex1}

  Our aim is to show that the transitivity condition of Assumption~\ref{asu:bigasu}.\ref{asu:transience} holds true under the third criterion provided below Assumption~\ref{asu:bigasu}.     We use the following auxiliary lemma, which is a consequence of superadditivity.

\begin{lem}
    \label{lem:superaddi}
    For all $z_0, \tilde z_0, z_1, \tilde z_1 \in \NN^p$,
    \begin{align*}
    \PP(Z_n\geq z_1+\tilde z_1 \, | \, Z_0\geq z_0+\tilde z_0)
    &\geq \PP(Z_n\geq z_1+\tilde z_1 \, | \, Z_0= z_0+\tilde z_0)\\
    &\geq \PP(Z_n\geq z_1 \, | \, Z_0= z_0)\times \PP(Z_n\geq \tilde z_1 \, | \, Z_0= \tilde z_0)\,.
    \end{align*}
\end{lem}

\begin{proof} First observe that the first inequality is a direct consequent of the fact that $\xi$ is an increasing function, so we deal only with the second inequality.  Consider $z_0,\tilde z_0, z_1, \tilde z_1 \in \NN^p$. We first treat the case $n=1$, and then proceed by induction. We have
\begin{align*}
\PP(Z_1 \geq z_1+\tilde z_1 &\mid Z_0=z_0+\tilde z_0) = \PP \left( \xi\left(\sum_{i=1}^p \sum_{k=1}^{z_{0,i}+\tilde z_{0,i}} V_{i,\cdot}^{(k,1)}\right)\geq z_1+\tilde z_1\right)\\
&\geq\PP\left(\xi\left(\sum_{i=1}^p \sum_{k=1}^{z_{0,i}} V_{i,\cdot}^{(k,1)}\right)+\xi\left(\sum_{i=1}^p \sum_{k=z_{0,i}+1}^{z_{0,i}+\tilde z_{0,i}} V_{i,\cdot}^{(k,1)}\right)\geq z_1+\tilde z_1\right)\\
&\geq \PP\left(\xi\left(\sum_{i=1}^p \sum_{k=1}^{z_{0,i}} V_{i,\cdot}^{(k,1)}\right)\geq z_1\, ,\xi\left(\sum_{i=1}^p \sum_{k=z_{0,i}+1}^{z_{0,i}+\tilde z_{0,i}} V_{i,\cdot}^{(k,1)}\right)\geq \tilde z_1\right)\\
&= \PP\left(\xi\left(\sum_{i=1}^p \sum_{k=1}^{z_{0,i}} V_{i,\cdot}^{(k,1)}\right)\geq z_1\right)\times \PP\left(\xi\left(\sum_{i=1}^p \sum_{k=1}^{\tilde z_{0,i}} V_{i,\cdot}^{(k,1)}\right)\geq \tilde z_1\right)\\
&=\PP(Z_1\geq z_1\mid Z_0=z_0)\times \PP(Z_1\geq \tilde z_1\mid Z_0=\tilde z_0).
\end{align*}

Observe that the first inequality in particular entails that if $(Z_n(z))_{n\in\NN}$ for $z\in\NN^p$ denotes a bGWbp with $Z_0=z$, then for all $z_0,\tilde z_0 \in \NN^p$, $Z_1(z_0+\tilde z_0)$ stochastically dominates $Z_1(z_0)+Z_1(\tilde z_0)$. Assume now that the result is true for some $n\in \NN$ and we prove it for $n+1$.
\begin{align*}
&\PP(Z_{n+1}\geq z_1+\tilde z_1 \mid Z_0=z_0+\tilde z_0)=\PP\left(\xi\left(\sum_{i=1}^p\sum_{k=1}^{Z_{n,i}(z_0+\tilde z_0)} V_{i,\cdot}^{(k,n+1)}\right)\geq z_1+\tilde z_1\right)\\
&\geq \PP\left(\xi\left(\sum_{i=1}^p\sum_{k=1}^{Z_{n,i}(z_0)}V_{i,\cdot}^{(k,n+1)}\right)\geq z_1\, ,\xi\left(\sum_{i=1}^p\sum_{k=Z_{n,i}(z_0)+1}^{Z_{n,i}(\tilde z_0)+Z_{n,i}(z_0)}V_{i,\cdot}^{(k,n+1)}\right)\geq \tilde z_1 \right)\\
&=\PP\left(\xi\left(\sum_{i=1}^p\sum_{k=1}^{Z_{n,i}(z_0)}V_{i,\cdot}^{(k,n+1)}\right)\geq z_1\right)\times \PP\left(\xi\left(\sum_{i=1}^p\sum_{k=Z_{n,i}(z_0)+1}^{Z_{n,i}(\tilde z_0)+Z_{n,i}(z_0)}V_{i,\cdot}^{(k,n+1)}\right)\geq \tilde z_1\right)\\
&=\PP(Z_{n+1}\geq z_1\mid Z_0=z_0)\times \PP(Z_{n+1}\geq \tilde z_1 \mid Z_0=\tilde z_0).
\end{align*}
which finishes the proof.
\end{proof}
We now prove that the third criterion provided below Assumption~\ref{asu:bigasu} leads to Assumption~\ref{asu:bigasu}.\ref{asu:transience} by contradiction. If this assumption is not satisfied, then there exits a non empty recurrent states class in $\NN^p\setminus\{0\}$. We can chose $z=(z_1,\dots,z_p)\in\NN^p\setminus\{0\}$ with the minimal size in its class, so that $\PP(|Z_n|\geq |z| \, | \, Z_0=z)=1$ for all $n\geq 1$. We prove that $2z$ is reachable from $z$, that is, there exist $n\geq 1$ such that $\PP(Z_n\geq 2\,z | Z_0=z)>0$. Once this is proved, we deduce from Lemma~\ref{lem:superaddi} that $\PP(|Z_n|\geq 2\,|z| | Z_0\geq 2\,z)\geq \PP(|Z_n|\geq |z| | Z_0=z)^2=1$, which implies that $\PP(\exists n,\, Z_n=z \, | \, Z_0\geq 2\,z)=0$ and contradicts the fact that $z$ is recurrent. We thus deduce that there is no recurrent state, except the absorbing state $\{0\}$, and Assumption~\ref{asu:bigasu}.\ref{asu:transience} is satisfied.

Let us prove that $2z$ is reachable from $z$ (in fact, we prove that every population can be doubled).
As the process is strongly primitive, for all $i\in \{1,\dots,p\}$, there exists $n_i$ such that for all $m\geq n_i$, $\PP(Z_{m,\ell}\geq 1 | Z_0=e_i)>0$. Let $m=\max_{1\leq i \leq p} n_i$, then for all $i\in \{1,\dots,p\}$,
\begin{equation}
\label{eq:ze_ell}
\PP(Z_m \geq e_\ell \, |\, Z_0=e_i)>0.
\end{equation}
Moreover, as $\PP(|Z_1|=2|Z_0=e_\ell)>0$, there exist $j_1,j_2 \in \{1,\dots,p\}$ such that 
\begin{equation}
\label{eq:doublepop}
\PP(Z_1\geq e_{j_1}+e_{j_2}\,|\, Z_0=e_\ell)>0.
\end{equation} 
By strongly primitive assumption, for $n=\max\{n_{j_1},n_{j_2}\}$, 
\[
\PP(Z_n \geq e_i \, |\, Z_0=e_{j_1})>0 \quad \text{and} \quad \PP(Z_n \geq e_i \, |\, Z_0=e_{j_2})>0,
\]
hence, by Lemma~\ref{lem:superaddi}, 
\begin{align}
\PP(Z_n \geq 2e_i \, |\, Z_0=e_{j_1}+e_{j_2}) 
&\geq \PP(Z_n \geq e_i \, |\, Z_0=e_{j_1})\,\PP(Z_n \geq e_i \, |\, Z_0=e_{j_2})
\label{eq:2z}
>0.
\end{align}
Finally,
\begin{align*}
\PP(Z_{n+m+1} \geq 2e_i \, | \, Z_0=e_i)
&\geq \PP(Z_{n+m+1} \geq 2e_i \, | \, Z_{m+1}\geq e_{j_1}+e_{j_2})
\\
&\quad \times   \PP(Z_{m+1} \geq e_{j_1}+e_{j_2} \, | \, Z_m \geq e_\ell)
\\
&\quad \times   \PP(Z_m \geq e_\ell \, | \, Z_0=e_i)
\end{align*}
then, from \eqref{eq:ze_ell}, \eqref{eq:doublepop}, \eqref{eq:2z} and the Markov property, $\PP(Z_{n+m+1} \geq 2e_i \, | \, Z_0=e_i)>0$ for all $i\in\{1,\dots,p\}$, and we conclude by Lemma~\ref{lem:superaddi} that $\PP(Z_{n+m+1} \geq 2z \, | \, Z_0=z)>0$.  

    \bibliographystyle{alpha}
    \bibliography{refs}
\end{document}